\RequirePackage{fix-cm}
\documentclass[11pt, a4paper, oneside, DIV11, final, pagebackref]{amsart}
\usepackage{fixltx2e}
\usepackage[latin1]{inputenc}
\usepackage[T1]{fontenc}
\usepackage{amssymb}
\usepackage[stretch=10]{microtype}
\usepackage{mathtools}
\usepackage[DIV11, headinclude=true]{typearea}
\usepackage{tikz}
\usetikzlibrary{intersections}
\usepackage[hidelinks, linktoc=all]{hyperref}

\providecommand{\href}[2]{#2}
\providecommand{\texorpdfstring}[2]{#1}
\providecommand*{\backref}{}
\providecommand*{\backrefalt}{}
\renewcommand*{\backref}[1]{}
\renewcommand*{\backrefalt}[4]{%
	\ifcase #1 %
	\or
	  Cited page~#2.
	\else
	  Cited pages~#2.
	\fi
}

\newcommand\MTkillspecial[1]{
  \bgroup
  \catcode`\&=9
  \let\\\relax%
  \scantokens{#1}%
  \egroup
}
\newcommand\DeclarePairedDelimiterMultiline[3]{
  \DeclarePairedDelimiter{#1}{#2}{#3}
  \reDeclarePairedDelimiterInnerWrapper{#1}{star}{
    \mathopen{##1\vphantom{\MTkillspecial{##2}}\kern-\nulldelimiterspace\right.}
    ##2
    \mathclose{\left.\kern-\nulldelimiterspace\vphantom{\MTkillspecial{##2}}##3}}
}

\newcommand{\gromprod}[3]{(#1 | #2)_{#3}}
\newcommand{\under}{\backslash}
\newcommand{\boA}{\mathcal{A}}
\newcommand{\boC}{\mathcal{C}}
\newcommand{\boE}{\mathcal{E}}
\newcommand{\boO}{\mathcal{O}}
\newcommand{\Sbb}{\mathbb{S}}
\newcommand{\Pbb}{\mathbb{P}}
\newcommand{\E}{\mathbb{E}}
\newcommand{\Z}{\mathbb{Z}}
\newcommand{\FF}{\mathbb{F}}
\newcommand{\HH}{\mathbb{H}}
\newcommand{\N}{\mathbb{N}}
\newcommand{\R}{\mathbb{R}}
\newcommand{\C}{\mathbb{C}}
\newcommand{\dd}{\mathop{}\!\mathrm{d}}
\DeclarePairedDelimiterMultiline{\abs}{\lvert}{\rvert}
\DeclarePairedDelimiter{\Card}{\lvert}{\rvert}
\DeclarePairedDelimiter{\paren}{(}{)}
\DeclarePairedDelimiter{\norm}{\lVert}{\rVert}
\newcommand{\st}{\::\:}

\DeclareMathOperator{\diam}{diam}
\DeclareMathOperator{\Supp}{Supp}

\newcommand{\coloneqq}{\mathrel{\mathop:}=}

\renewcommand{\epsilon}{\varepsilon}

\renewcommand{\phi}{\varphi}
\renewcommand{\leq}{\leqslant}
\renewcommand{\geq}{\geqslant}

\newtheorem{thm}{Theorem}[section]
\newtheorem{prop}[thm]{Proposition}
\newtheorem{definition}[thm]{Definition}
\newtheorem{lem}[thm]{Lemma}

\theoremstyle{definition}
\newtheorem{ex}[thm]{Example}
\newtheorem{rmk}[thm]{Remark}

\numberwithin{equation}{section}

\title{Entropy and drift in word hyperbolic groups}

\author{S\'ebastien Gou\"ezel, Fr\'ed\'eric Math\'eus, Fran\c{c}ois Maucourant}

\address{S\'ebastien Gou\"ezel, IRMAR, Univ.\ Rennes 1,
35042 Rennes Cedex, France} \email{sebastien.gouezel@univ-rennes1.fr}
\address{Fr\'ed\'eric Math\'eus, Univ.\ Bret.\ Sud, L.M.B.A., UMR 6205,
BP 573, 56017 Vannes, France} \email{Frederic.Matheus@univ-ubs.fr}
\address{Fran\c{c}ois Maucourant, IRMAR, Univ.\ Rennes 1,
35042 Rennes Cedex, France}
\email{francois.maucourant@univ-rennes1.fr}

\date{January 20, 2015}

\begin{document}
\begin{abstract}
The fundamental inequality of Guivarc'h relates the entropy and the drift
of random walks on groups. It is strict if and only if the random walk does
not behave like the uniform measure on balls. We prove that, in any
nonelementary hyperbolic group which is not virtually free, endowed with a
word distance, the fundamental inequality is strict for symmetric measures
with finite support, uniformly for measures with a given support. This
answers a conjecture of S.~Lalley. For admissible measures, this is proved
using previous results of Ancona and Blach\`ere-Ha\"{\i}ssinsky-Mathieu. For
non-admissible measures, this follows from a counting result, interesting
in its own right: we show that, in any infinite index subgroup, the number
of non-distorted points is exponentially small. The uniformity is obtained
by studying the behavior of measures that degenerate towards a measure
supported on an elementary subgroup.
\end{abstract}

\maketitle

\section{Main results}

Let $\Gamma$ be a finitely generated infinite group. Although the
following discussion makes sense in a much broader context, we will
assume that $\Gamma$ is hyperbolic since all results of this article
are devoted to this setting. There are two natural ways to construct
random elements in $\Gamma$:
\begin{itemize}
\item Let $d$ be a proper left-invariant distance on $\Gamma$
    (for instance a word distance). For large $n$, one can pick
    an element at random with respect to the uniform measure
    $\rho_n$ on the ball $B_n=B(e,n)$ (where $e$ denotes the
    identity of $\Gamma$).
\item Let $\mu$ be a probability measure on $\Gamma$. For large
    $n$, one can pick an element at random with respect to the
    measure $\mu^{*n}$ (the $n$-th convolution of the measure
    $\mu$). Equivalently, let $g_1,g_2,\dotsc$ be a sequence of
    random elements of $\Gamma$ that are distributed
    independently according to $\mu$. Form the random walk
    $X_n=g_1\dotsm g_n$. Then the distribution of $X_n$ is
    $\mu^{*n}$.
\end{itemize}
From a theoretical point of view, these methods share a lot of properties.
From a computational point of view, the second method is much easier to
implement in general groups since it does not require the computation of the
ball $B_n$ (note however that, in hyperbolic groups, simulating the uniform
measure is very easy thanks to the automatic structure of the group). It is
therefore of interest to find probability measures $\mu$ such that these two
methods give equivalent results, in a sense that will be made precise below.
This is the main question of Vershik in~\cite{vershik}. In free groups (with
the word distance coming from the usual set of generators), everything can be
computed: if $\mu$ is the uniform measure on the generators, then $\mu^{*n}$
and $\rho_n$ behave essentially in the same way. The situation is the same in
free products of finite groups, again thanks to the underlying tree
structure. However, in more complicated groups, explicit computations are
essentially impossible, and it is expected that the methods always differ.
Our main result confirms this intuition in a special class of groups: In
hyperbolic groups which are not virtually free (i.e., there is no finite
index free subgroup), if $d$ is a word distance, the two methods are always
different, in a precise quantitative way.

\begin{rmk}
We emphasize that the question really depends on the choice of the
distance $d$, since the shape of the balls $B_n$ depends on $d$. For
instance, for any symmetric probability measure $\mu$ on $\Gamma$
whose support is finite and generates $\Gamma$, there exists a
distance $d$ (called the Green distance, see~\cite{BHM_2}) for which
the measures $\rho_n$ and $\mu^{*n}$ behave in the same way. A famous
open problem (to which our methods do not apply) is to understand
what happens when $\Gamma$ acts cocompactly on the hyperbolic space
$\HH^k$, and the distance $d$ is given by $d(e,\gamma) =
d_{\HH^k}(O,\gamma\cdot O)$ where $O$ is a base point in $\HH^k$. In
this case, it is also expected that the two methods are always
different. Here are the main partial results in this context:
\begin{enumerate}
\item The two methods are different for some symmetric measures
    with finite support (\cite{leprince}, see also
    Theorem~\ref{thm:infinite_limit} below).
\item If, instead of a cocompact lattice, one considers a lattice
    with cusps, the two methods are always
    different~\cite{guivarch_lejan}.
\item If, instead of a lattice, one considers a nice dense
    subgroup, there exist symmetric measures with finite support
    for which the two methods are
    equivalent~\cite{bourgain_spectral_SL2}.
\end{enumerate}
This question also makes sense in continuous time, for negatively
curved manifolds. A conjecture of Sullivan asserts that, in this
setting, the two methods coincide if and only if the manifold is
locally symmetric, see~\cite{ledrappier_rigidity}.
\end{rmk}

One can give several meanings to the question ``are the two methods
equivalent?'' Let us first discuss an interpretation in terms of behavior at
infinity. The measures $\mu^{*n}$ converge in the geometric compactification
$\Gamma \cup \partial \Gamma$ to a measure $\mu_\infty$, supported on the
boundary, called the exit measure of the random walk, or its stationary
measure. Geometrically, the random walk $(X_n)_{n\geq 1}$ converges almost
surely to a random point on the boundary $\partial\Gamma$, the measure
$\mu_\infty$ is its distribution. On the other hand, let $\rho_\infty$ be the
Patterson-Sullivan measure on $\partial\Gamma$ associated to the distance
$d$, constructed in~\cite{coornaert} in this context. One should think of it
as the uniform measure on the boundary (it is equivalent to the Hausdorff
measure of maximal dimension on the boundary, for any visual distance coming
from $d$). The measures $\rho_n$ do not always converge to $\rho_\infty$, but
all their limit points are equivalent to $\rho_\infty$, with a density
bounded from above and from below (this follows from the arguments
of~\cite{coornaert}, see Lemma~\ref{lem:converge_PS} below). A version of the
question is then to ask if the measures $\mu_\infty$ and $\rho_\infty$ are
mutually singular: in this case, the random walk mainly visits parts of the
groups that are not important from the point of view of the uniform measure.

Another version of the same question is quantitative: Does the random walk
visit parts of the groups that are exponentially negligible from the point of
view of the uniform measure? This is made precise through the notions of
drift and entropy. Define
  \begin{equation}
  \label{defLmu}
  L(\mu) = \sum_{g\in \Gamma} \mu(g) \abs{g},\quad H(\mu) = \sum_{g\in \Gamma} \mu(g) (-\log \mu(g)),
  \end{equation}
where $\abs{g}=d(e,g)$. The quantity $L(\mu)$ is the average distance of an
element to the identity. The quantity $H(\mu)$, called the time one entropy
of $\mu$, is the average logarithmic weight of the points. They can both be
finite or infinite. The functions $L$ and $H$ both behave in a subadditive
way with respect to convolution: $L(\mu_1 * \mu_2)\leq L(\mu_1) + L(\mu_2)$
and $H(\mu_1*\mu_2) \leq H(\mu_1)+H(\mu_2)$. It follows that the sequences
$L(\mu^{*n})$ and $H(\mu^{*n})$ are subadditive. Hence, the following
quantities are well defined:
  \begin{equation}
  \label{eq:def_h_ell}
  \ell(\mu) = \lim L(\mu^{*n})/n,\quad h(\mu) = \lim H(\mu^{*n})/n.
  \end{equation}
They are called respectively the drift and the asymptotic entropy of
the random walk. They also admit characterizations along typical
trajectories. If $L(\mu)$ is finite, then almost surely $\ell(\mu) =
\lim \abs{X_n}/n$. In the same way, if $H(\mu)$ is finite, then
almost surely $h(\mu) = \lim (-\log \mu^{*n}(X_n))/n$. The most
intuitive characterization of the entropy is probably the following
one: at time $n$, the random walk is essentially supported by
$e^{h(\mu)n}$ points (see Lemma~\ref{lem:charac_entropy} for a
precise statement). Let us also define the exponential growth rate of
the group with respect to $d$, i.e.,
  \begin{equation}
  \label{eq:fundam}
  v= \liminf_{n\to \infty} \frac{\log \Card{B_n}}{n},
  \end{equation}
where $B_n$ is the ball of radius $n$ around $e$. In hyperbolic
groups, it satisfies the apparently stronger inequality $C^{-1}
e^{nv} \leq \Card{B_n} \leq Ce^{nv}$, by~\cite{coornaert}. For large
$n$, most points for $\mu^{*n}$ are contained in a ball $B_{
(1+\epsilon)\ell n}$, which has cardinality at most $e^{
(1+2\epsilon)\ell nv}$. Since the random walk at time $n$ essentially
visits $e^{h n}$ points, we deduce the fundamental inequality of
Guivarc'h~\cite{guivarch_fundamental_ineq}
  \begin{equation*}
  h \leq \ell v.
  \end{equation*}
If this inequality is an equality, this means that the walk visits most parts
of the group. Otherwise, it is concentrated in an exponentially small subset.
Another version of our main question is therefore: Is the inequality $h\leq
\ell v$ strict?

In hyperbolic groups, it turns out that the two versions of the question are
equivalent, at least for finitely supported measures, and that they also have
a geometric interpretation in terms of Hausdorff dimension. If $\mu$ is a
probability measure on a group, we write $\Gamma_\mu^+$ for the semigroup
generated by the support of $\mu$, and $\Gamma_\mu$ for the group it
generates. When $\mu$ is symmetric, they coincide. We say that $\mu$ is
admissible if $\Gamma_\mu^+=\Gamma$. The following result is Corollary~1.4
and Theorem~1.5 in~\cite{BHM_2} (see also~\cite{haissinsky}) when the measure
is symmetric, and is proved in~\cite{tanaka} when $\mu$ is not necessarily
symmetric and $d$ is a word distance.

\begin{thm}
\label{thm:BHM} Let $\Gamma$ be a non-elementary hyperbolic group, endowed
with a left-invariant distance $d$ which is hyperbolic and quasi-isometric to
a word distance. Let $v$ be the exponential growth rate of $(\Gamma,d)$. Let
$d_{\partial\Gamma}$ be a visual distance on $\partial \Gamma$ associated to
$d$. Consider an admissible probability measure $\mu$ on $\Gamma$, with
finite support. Assume additionally either that the measure $\mu$ is
symmetric, or that the distance $d$ is a word distance. The following
conditions are equivalent:
\begin{enumerate}
\item The equality $h=\ell v$ holds.
\item The Hausdorff dimension of the exit measure $\mu_\infty$ on
    $(\partial \Gamma, d_{\partial \Gamma})$ is equal to the
    Hausdorff dimension of this space.
\item The measure $\mu_\infty$ is equivalent to the
    Patterson-Sullivan measure $\rho_\infty$.
\item The measure $\mu_\infty$ is equivalent to the
    Patterson-Sullivan measure $\rho_\infty$, with density
    bounded from above and from below.
\item There exists $C>0$ such that, for any $g \in \Gamma$,
      \begin{equation*}
      \abs{ vd(e,g) - d_\mu(e,g)}\leq C,
      \end{equation*}
    where $d_\mu$ is the ``Green distance'' associated to $\mu$, i.e.,
    $d_\mu(e,g) = -\log \Pbb(\exists n, X_n=g)$ where $X_n$ is the random
    walk given by $\mu$ starting from the identity (it is an asymmetric
    distance in general, and a genuine distance if $\mu$ is symmetric).
\end{enumerate}
\end{thm}
The different statements in this theorem go from the weakest to the
strongest: since entropy is an asymptotic quantity, an assumption on
$h$ seems to allow subexponential fluctuations, so the assumption (1)
is rather weak. On the other hand, (3) says that two measures are
equivalent, so most points are controlled. Finally, in (5), all
points are uniformly controlled. The equivalence between these
statements is a strong rigidity theorem. The equivalence between (1)
and (2) follows from a formula for the respective dimensions. The
definition of a visual distance at infinity $d_{\partial\Gamma}$
involves a small parameter $\epsilon$. In terms of this parameter,
one has $HD(\mu_\infty)=h/(\epsilon \ell)$ and $HD(\rho_\infty) =
HD(\partial\Gamma) = v/\epsilon$, so that these dimensions coincide
if and only if $h=\ell v$.

In this theorem, the finite support assumption can be weakened to an
assumption of superexponential moment (i.e., for all $M>0$, $\sum_{g\in
\Gamma} \mu(g) e^{M \abs{g}} < \infty$), thanks
to~\cite{gouezel_infinite_support}. The assumption that $\mu$ is symmetric or
that $d$ is a word distance is probably not necessary. However, the most
important assumption in Theorem~\ref{thm:BHM} is admissibility: it ensures
that the random walk can see the geometry of the whole group (which is
hyperbolic). For a random walk living in a strict (maybe distorted) subgroup,
one would not be expecting the same nice behavior.

\medskip

Our main theorem follows. It states that, in hyperbolic groups which are not
virtually free, endowed with a word distance, the different equivalent
conditions of Theorem~\ref{thm:BHM} are never satisfied, uniformly on
measures with a fixed support.

\begin{thm}
\label{thm:principal}
Let $\Gamma$ be a hyperbolic group which is not virtually free,
endowed with a word distance $d$. Let $\Sigma$ be a finite subset of
$\Gamma$. There exists $c<1$ such that, for any symmetric probability
measure $\mu$ supported in $\Sigma$,
  \begin{equation*}
  h(\mu)\leq c \ell(\mu) v,
  \end{equation*}
where $v$ is the exponential growth rate of balls in $(\Gamma,d)$.
\end{thm}
This theorem gives a positive answer to a conjecture of S.~Lalley~\cite[slide
16]{lalley}. In the language of Vershik~\cite{vershik}, this theorem says
that no finite subset of $\Gamma$ is extremal. On the other hand, if one lets
$\Sigma$ grow, $h/\ell$ can converge to $v$:

\begin{thm}
\label{thm:tends_to_v}
Let $\Gamma$ be a hyperbolic group, endowed with a left invariant
distance $d$ which is hyperbolic and quasi-isometric to a word
distance. Let $\rho_i$ be the uniform measure on the ball of radius
$i$. Then $h(\rho_i)/\ell(\rho_i) \to v$, where $v$ is the
exponential growth rate of balls in $(\Gamma,d)$.
\end{thm}
More precisely, we prove that $\ell(\rho_i)\sim i$ and $h(\rho_i)\sim iv$.
The only difficulty is to prove the lower bound on $h(\rho_i)$: since $h$ is
defined in~\eqref{eq:def_h_ell} using a subadditive sequence, upper bounds
are automatic, but to get lower bounds one should show that additional
cancellations do not happen later on. This difficulty already appears
in~\cite{erschler_kaim}, where the authors prove that the entropy depends
continuously on the measure. Our proof of Theorem~\ref{thm:tends_to_v}, given
in Paragraph~\ref{subsec:tends_to_v}, also applies to this situation and
gives a new proof of their result, under slightly weaker assumptions. There
is nothing special about the uniform measure on balls, our proof also gives
the same conclusion for the uniform measure on spheres, or for the measures
$\sum e^{-s\abs{g}}\delta_g/\sum e^{-s \abs{g}}$ when $s \searrow v$.

\medskip

Our main result is Theorem~\ref{thm:principal}. It is a consequence
of the three following results. Since their main aim is
Theorem~\ref{thm:principal}, they are designed to handle finitely
supported symmetric measures. However, these theorems are all valid
under weaker assumptions, which we specify in the statements as they
carry along implicit information on the techniques used in the
proofs.

The first result deals with admissible (or virtually admissible)
measures.
\begin{thm}
\label{thm:hlv_indice_fini} Let $\Gamma$ be a hyperbolic group which is not
virtually free, endowed with a word distance. Let $\mu$ be a probability
measure with a superexponential moment, such that $\Gamma_\mu^+$ is a finite
index subgroup of $\Gamma$. Then $h(\mu)<\ell(\mu)v$.
\end{thm}

The second result deals with non-admissible measures.
\begin{thm}
\label{thm:hlv_indice_infini}
Let $\Gamma$ be a hyperbolic group endowed with a word distance. Let
$\mu$ be a probability measure with a moment of order $1$ (i.e.,
$L(\mu)<\infty$). Assume that $\ell(\mu)>0$ and that $\Gamma_\mu$ has
infinite index in $\Gamma$. Then $h(\mu)<\ell(\mu) v$.
\end{thm}

Finally, the third result is a kind of continuity statement, to get the
uniformity.
\begin{thm}
\label{thm:hsurl_atteint}
Let $\Gamma$ be a hyperbolic group, endowed with a left-invariant
distance which is hyperbolic and quasi-isometric to a word distance.
Let $\Sigma$ be a subset of $\Gamma$ which does not generate an
elementary subgroup. There exists a probability measure $\mu_\Sigma$
with finite support such that $\ell(\mu_\Sigma)>0$ and
  \begin{equation*}
  \sup \{ h(\mu)/ \ell(\mu) \st \mu\text{ probability}, \Supp(\mu) \subset \Sigma, \ell(\mu)>0\}
  =h(\mu_\Sigma)/\ell(\mu_\Sigma).
  \end{equation*}
The same statement holds if the maximum is taken over symmetric
probability measures, the resulting maximizing measure being
symmetric.
\end{thm}

Theorem~\ref{thm:principal} is a consequence of these three
statements.
\begin{proof}[Proof of Theorem~\ref{thm:principal} using the three auxiliary theorems]
As in the statement of the theorem, consider a finite subset $\Sigma$ of
$\Gamma$. If $\Sigma$ generates an elementary subgroup of $\Gamma$, all
measures supported on $\Sigma$ have zero entropy. Hence, one can take $c=0$
in the statement of the theorem. Otherwise, by
Theorem~\ref{thm:hsurl_atteint}, there exists a symmetric measure
$\mu_\Sigma$ with finite support that maximizes the quantity
$h(\mu)/\ell(\mu)$ over $\mu$ symmetric supported by $\Sigma$. If
$\Gamma_{\mu_\Sigma}=\Gamma_{\mu_\Sigma}^+$ has finite index,
$h(\mu_\Sigma)/\ell(\mu_\Sigma)<v$ by Theorem~\ref{thm:hlv_indice_fini}. If
it has infinite index, the same conclusion follows from
Theorem~\ref{thm:hlv_indice_infini}.
\end{proof}

The three auxiliary theorems are non-trivial. Their proofs are
independent, and use completely different tools. Here are some
comments about them.
\begin{itemize}
\item At first sight, Theorem~\ref{thm:hlv_indice_fini} seems to be the
    most delicate (this is the only one with the assumption that $\Gamma$
    is not virtually free). However, this is also the setting that has been
    mostly studied in the literature. Hence, we may use several known
    results, including most notably results of Ancona~\cite{ancona}, of
    Blach\`ere, Ha\"{\i}ssinsky and Mathieu~\cite{BHM_2} and Tanaka~\cite{tanaka}
    (Theorem~\ref{thm:BHM} above) and of Izumi, Neshveyev and
    Okayasu~\cite{izumi_hyperbolic} on rigidity results for cocycles. The
    proof relies mainly on the fact that the word distance is integer
    valued, contrary to the Green distance (more precisely, we use the fact
    that the stable translation length of hyperbolic elements is rational
    with bounded denominator).
\item In Theorem~\ref{thm:hlv_indice_infini}, the difficulty comes from the
    lack of information on the subgroup $\Gamma_\mu$. If it has good
    geometric properties (for instance if it is quasi-convex), one may use
    the same kind of techniques as for Theorem~\ref{thm:hlv_indice_fini}.
    Otherwise, the random walk does not really see the hyperbolicity of the
    ambient group. The fundamental inequality always gives $h\leq \ell
    v_{\Gamma_\mu}$, where $v_{\Gamma_\mu}$ is the growth rate of the
    subgroup $\Gamma_\mu$ (for the initial word distance on $\Gamma$). If
    $v_{\Gamma_\mu}<v$, the result follows. Unfortunately, there exist
    non-quasi-convex subgroups of some hyperbolic groups with the same
    growth as the ambient group. However, a random walk does not typically
    visit all points of $\Gamma_\mu$, it concentrates on those points that
    are not distorted (i.e., their distances to the identity in $\Gamma$
    and $\Gamma_\mu$ are comparable). To prove
    Theorem~\ref{thm:hlv_indice_infini}, we will show that in any infinite
    index subgroup of a hyperbolic group, the number of non-distorted
    points is exponentially smaller than $e^{nv}$.
\item Theorem~\ref{thm:hsurl_atteint} is less simple than it may
    seem at first sight: it does not claim that $\mu_\Sigma$ is
    supported by $\Sigma$, and indeed this is not the case in
    general (see Example~\ref{ex:not_in_Sigma}). Hence, the proof
    is not a simple continuity argument: We need to understand
    precisely the behavior of sequences of measures that
    degenerate towards a measure supported on an elementary
    subgroup. The proof will show that $\mu_\Sigma$ is supported
    by $K\cdot (\Sigma\cup \{e\}) \cdot K$, where $K$ is a finite
    subgroup generated by some elements in $\Sigma$.
\end{itemize}

A natural question is whether Theorem~\ref{thm:principal} holds for
non-symmetric measures. For admissible measures, (i.e.,
$\Gamma_\mu^+=\Gamma$), Theorem~\ref{thm:hlv_indice_fini} holds. For
non-symmetric measures such that $\Gamma_\mu$ has infinite index,
Theorem~\ref{thm:hlv_indice_infini} applies directly. However, since
$\Gamma_\mu\not=\Gamma_\mu^+$ for general non-symmetric measures, there is
another case to consider: the case of measures $\mu$ such that
$\Gamma_\mu=\Gamma$ (or $\Gamma_\mu$ has finite index in $\Gamma$), but
$\Gamma_\mu^+$ is much smaller than $\Gamma$. In this case, it seems that our
arguments do not suffice. We give in Section~\ref{sec:non-symmetric} two
examples illustrating the new difficulties:
\begin{enumerate}
\item One can not rely on growth arguments, as for
    Theorem~\ref{thm:hlv_indice_infini}. Indeed, there are
    subsemigroups $\Lambda^+$ with bad asymptotic behavior, for
    instance such that $\liminf \Card{B_n \cap
    \Lambda^+}/\Card{B_n}=0$ and $\limsup \Card{B_n \cap
    \Lambda^+}/\Card{B_n}>0$.
\item The arguments of Theorem~\ref{thm:hlv_indice_fini} work for finitely
    supported measures, or for measures with a superexponential moment, but
    also more generally for measures with a nice geometric behavior (they
    should satisfy so-called Ancona inequalities). In the non-symmetric
    situation, we give in Proposition~\ref{prop:nonsymmetric_equal}
    explicit examples of (non-admissible) measures with an exponential
    moment and a very nice geometric behavior, and such that nevertheless
    $h=\ell v$. So, arguments similar to those of
    Theorem~\ref{thm:hlv_indice_fini} can not suffice, one needs a new
    argument that distinguishes in a finer way between measures with finite
    support and measures with infinite support.
\end{enumerate}

This article is organized as follows. In Section~\ref{sec:general},
we give more details on the notions of hyperbolic group, drift and
entropy. We also prove Theorem~\ref{thm:tends_to_v} on the asymptotic
entropy and drift of the uniform measure on large balls. The
following three sections are then devoted to the proofs of the three
auxiliary theorems. Finally, we describe in
Section~\ref{sec:non-symmetric} what can happen in the non-symmetric
setting. In  particular, we show that in any torsion-free group with
infinitely many ends, there exist (non-admissible, non-symmetric)
measures with an exponential moment satisfying $h=\ell v$.

\section{General properties of entropy and drift in hyperbolic
groups}
\label{sec:general}

\subsection{Hyperbolic spaces}

In this paragraph, we recall classical properties of hyperbolic
spaces. See for instance~\cite{ghys_hyperbolique}
or~\cite{bridson_haefliger}.

Consider a metric space $(X,d)$. The Gromov product of two points
$y,y'\in X$, based at $x_0\in X$, is by definition
  \begin{equation}
  \label{def:gromov_product}
  \gromprod{y}{y'}{x_0} = (1/2) [ d(x_0,y)+d(x_0,y')-d(y,y') ].
  \end{equation}
The space $(X,d)$ is hyperbolic if there exists $\delta\geq 0$ such
that, for any $x_0,y_1,y_2,y_3$, the following inequality holds:
  \begin{equation*}
  \gromprod{y_1}{y_3}{x_0}\geq \min( \gromprod{y_1}{y_2}{x_0}, \gromprod{y_2}{y_3}{x_0}) -\delta.
  \end{equation*}

The main intuition to have is that, in hyperbolic spaces,
configurations of finitely many points look like configurations in
trees: for any $k$, for any subset $F$ of $X$ with cardinality at
most $k$, there exists a map $\Phi$ from $F$ to a tree such that, for
all $x,y\in F$,
  \begin{equation*}
  d(x,y)-2k\delta \leq d(\Phi(x),\Phi(y)) \leq d(x,y).
  \end{equation*}
Hence, a lot of distance computations can be reduced to equivalent
computations in trees (which are essentially combinatorial), up to a
bounded error. Up to $\delta$, the Gromov product
$\gromprod{y}{y'}{x_0}$ is, in the approximating tree, the length of
the part that is common to the geodesics from $x_0$ to $y$ and from
$x_0$ to $y'$.

A space $(X,d)$ is geodesic if there exists a geodesic between any pair of
points. For such spaces, there is a convenient characterization of
hyperbolicity. A geodesic space $(X,d)$ is hyperbolic if and only if there
exists $\delta\geq 0$ such that its geodesic triangles are $\delta$-thin,
i.e., each side is included in the $\delta$-neighborhood of the union of the
two other sides.

Assume that $(X,d_X)$ and $(Y,d_Y)$ are two geodesic metric spaces,
and that they are quasi-isometric. If $(X,d_X)$ is hyperbolic, then
so is $(Y,d_Y)$. Note however that this equivalence only holds for
geodesic spaces.

Let $(X,d)$ be a geodesic hyperbolic metric space. A subset $Y$ of
$X$ is quasi-convex if there exists a constant $C$ such that, for any
$y,y'\in Y$, the geodesics from $y$ to $y'$ stay in the
$C$-neighborhood of $Y$.

We will sometimes encounter hyperbolic spaces which are not geodesic,
but only quasi-geodesic: there exist constants $C>0$ and $\lambda$
such that any two points can be joined by a
$(\lambda,C)$-quasi-geodesic, i.e., a map $f$ from a real interval to
$X$ such that $\lambda^{-1} \abs{t'-t}-C \leq d(f(t), f(t')) \leq
\lambda \abs{t'-t} + C$. When the space is geodesic, a quasi-geodesic
stays a bounded distance away from a true geodesic. Most properties
that hold or can be defined using geodesics (for instance the notion
of quasi-convexity) can be extended to this setting, simply replacing
geodesics with quasi-geodesics in the statements.

\medskip

Let $(X,d)$ be a proper geodesic hyperbolic space. Its boundary at
infinity $\partial X$ is by definition the set of geodesics
originating from a base point $x_0$, where two such geodesics are
identified if they remain a bounded distance away. It is a compact
space, which does not depend on $x_0$. The space $X \cup
\partial X$ is also compact. If $X$ is only quasi-geodesic, all these
definitions extend using quasi-geodesics instead of geodesics.

Any isometry (or, more generally, quasi-isometry) of a hyperbolic
space extends continuously to its boundary, giving a homeomorphism of
$\partial X$.

The Gromov product may be extended to $X\cup \partial X$: we define
$\gromprod{\xi}{\eta}{x_0}$ as the infimum limit of
$\gromprod{x_n}{y_n}{x_0}$ for $x_n$ and $y_n$ converging
respectively to $\xi$ and $\eta$. The choice to take the infimum is
arbitrary, one could also take the supremum or any accumulation
point, those quantities differ by at most a constant only depending
on $\delta$. Hence, one should think of the Gromov product at
infinity to be canonically defined only up to an additive constant.
Heuristically, $\gromprod{\xi}{\eta}{x_0}$ is the time after which
two geodesics from $x_0$ to $\xi$ and to $\eta$ start diverging.

Let $(X,d)$ be a proper geodesic (or quasi-geodesic) hyperbolic space. For
any small enough $\epsilon>0$, one may define a visual distance $d_{\partial
X,\epsilon}$ on $\partial X$ such that $d_{\partial X,\epsilon}(\xi,\eta)
\asymp e^{-\epsilon \gromprod{\xi}{\eta}{x_0}}$ (meaning that the ratio
between these quantities is uniformly bounded from above and from below).

\medskip

Let $(X,d)$ be a proper hyperbolic metric space. One can define
another boundary of $X$, the Busemann boundary (or horoboundary), as
follows. Let $x_0$ be a fixed basepoint in $X$. To $x \in X$, one
associates its horofunction $h_x(y) = d(y,x)-d(x_0,x)$, normalized so
that $h_x(x_0)=0$. The map $\Phi:x\mapsto h_x$ is an embedding of $X$
into the space of $1$-Lipschitz functions on $X$, with the topology
of uniform convergence on compact sets. The horoboundary is obtained
by taking the closure of $\Phi(X)$. In other words, a sequence $x_n
\in X$ converges to a boundary point if $h_{x_n}(y)$ converges,
uniformly on compact sets. Its limit is the horofunction $h_\xi$
associated to the corresponding boundary point $\xi$ (it is also
called the Busemann function associated to $\xi$). We denote by
$\partial_B X$ the Busemann boundary of $X$. There is a continuous
projection $\pi_B : \partial_B X\to \partial X$, which is onto but
not injective in general. The boundary $\partial_B X$ is rather
sensitive to fine scale details of the distance $d$, while $\partial
X$ only depends on its quasi-isometry class.

Any isometry $\phi$ of $X$ acts on horofunctions, by the formula
$h_{\phi(x)}(y) = h_x(\phi^{-1}y) - h_x(\phi^{-1} x_0)$. This implies
that $\phi$ extends to a homeomorphism on $\partial_B X$, given by
the same formula $h_{\phi(\xi)}(y) = h_\xi(\phi^{-1}y) -
h_\xi(\phi^{-1} x_0)$. Note that, contrary to the action on the
geometric boundary, this only works for isometries of $X$, not
quasi-isometries.

\subsection{Hyperbolic groups}

Let $\Gamma$ be a finitely generated group, with a finite symmetric
generating set $S$. Denote by $d=d_S$ the corresponding word
distance. The group $\Gamma$ is hyperbolic if the metric space
$(\Gamma,d_S)$ is hyperbolic. Since hyperbolicity is invariant under
quasi-isometry for geodesic spaces, this notion does not depend on
the choice of the generating set $S$. However, if one considers
another left-invariant distance on $\Gamma$ which is equivalent to
$d_S$ but not geodesic, its hyperbolicity is not automatic. Hence,
one should postulate its hyperbolicity if it is needed, as in the
statement of Theorem~\ref{thm:BHM}. We say that the pair $(\Gamma,d)$
is a metric hyperbolic group if the group $\Gamma$ is hyperbolic for
one (or, equivalently, for any) word distance, and if the distance
$d$ is left-invariant, hyperbolic, and quasi-isometric to one (or
equivalently, any) word distance. Such a distance $d$ does not have
to be geodesic, but it is quasi-geodesic since geodesics for a given
word distance form a system of quasi-geodesics for $d$, going from
any point to any point.

Let $(\Gamma,d)$ be a metric hyperbolic group. The
left-multiplication by elements of $\Gamma$ is isometric. Hence,
$\Gamma$ acts by homeomorphisms on its compactifications
$\Gamma\cup\partial \Gamma$ and $\Gamma\cup\partial_B \Gamma$.
Moreover, any infinite order element $g\in \Gamma$ acts
hyperbolically on $\Gamma\cup\partial \Gamma$: it has two fixed
points at infinity $g^-$ and $g^+$, the points in $\Gamma \cup
\partial\Gamma \setminus\{g^-\}$ are attracted to $g^+$ by forward
iteration of $g$, and the points in $\Gamma \cup \partial\Gamma
\setminus\{g^+\}$ are attracted to $g^-$ by backward iteration of
$g$.

\begin{definition}
\label{def:cocycle}
Consider an action of a group $\Gamma$ on a space $Z$. A function
$c:\Gamma\times Z \to \R$ is a cocycle if, for any $g,h\in \Gamma$
and any $\xi\in Z$,
  \begin{equation}
  \label{eq:cocycle}
  c(gh,\xi) = c(g,h\xi) + c(h, \xi).
  \end{equation}
The cocycle is H\"older-continuous if $Z$ is a metric space and each
function $\xi \mapsto c(g,\xi)$ is H\"older-continuous.
\end{definition}
There is a choice to be made in the definition of cocycles, since one may
compose with $g$ or $g^{-1}$. Our definition is the most customary. With this
definition, the map $c_B:\Gamma\times \partial_B \Gamma \to \R$ given by
$c_B(g,\xi)=h_\xi(g^{-1})$ is a cocycle, called the Busemann cocycle.

A subgroup $H$ of $\Gamma$ is nonelementary if its action on $\partial\Gamma$
does not fix a finite set. Equivalently, $H$ is not virtually the trivial
group or $\Z$. We say that a probability measure $\mu$ on $\Gamma$ is
nonelementary if the subgroup $\Gamma_\mu$ generated by its support is itself
nonelementary.

Let $\mu$ be a probability measure on $\Gamma$. Since $\Gamma$ acts by
homeomorphisms on the compact space $\partial \Gamma$, it admits a stationary
measure: there exists a probability measure $\nu$ on $\partial \Gamma$ such
that $\mu* \nu=\nu$, i.e., $\sum_{g\in \Gamma} \mu(g) g_* \nu = \nu$. If
$\mu$ is nonelementary, this measure is unique, and has no atom
(see~\cite{kaimanovich_poisson}). It is also the exit measure of the
corresponding random walk $X_n=g_1\dotsm g_n$: almost every trajectory
$X_n(\omega)$ converges to a point $X_\infty(\omega) \in
\partial\Gamma$, and moreover the distribution of $X_\infty$ is
precisely $\nu$.

In the same way, since $\Gamma$ acts on $\partial_B \Gamma$, it
admits a stationary measure $\nu_B$ there. This measure is not unique
in general, even if $\mu$ is nonelementary. However, all such
measures project under $\pi_B$ to the unique stationary measure on
$\partial\Gamma$.

\subsection{The drift}

Let $(\Gamma,d)$ be a metric hyperbolic group. Consider a probability
measure $\mu$ on $\Gamma$, with finite first moment $L(\mu)$ (defined
in~\eqref{defLmu}). The drift of the random walk has been defined
in~\eqref{eq:def_h_ell} as $\ell(\mu)=\lim L(\mu^{*n})/n$. Let
$X_n=g_1\dotsm g_n$ be the position at time $n$ of the random walk
generated by $\mu$ (where the $g_i$ are independent and distributed
according to $\mu$). Then, almost surely, $\ell(\mu) = \lim
\abs{X_n}/n$.

The drift also admits a description in terms of the Busemann boundary. The
following result is well-known (compare with~\cite[Theorem
18]{karlsson_ledrappier_noncommutative}).
\begin{prop}
\label{prop:express_drift}
Let $(\Gamma,d)$ be a metric hyperbolic group. Let $\mu$ be a
nonelementary probability measure on $\Gamma$ with finite first
moment. Let $\nu_B$ be a $\mu$-stationary measure on $\partial_B
\Gamma$. Then
  \begin{equation}
  \label{eq:asymp_ell}
  \ell(\mu) =\int_{\Gamma\times\partial_B \Gamma} c_B(g,\xi)\dd\mu(g) \dd\nu_B(\xi).
  \end{equation}
\end{prop}
\begin{proof}
Let $X_n$ be the position of the random walk at time $n$. Using the
cocycle property of the Busemann cocycle, we have
  \begin{align*}
  \int c_B(X_n(\omega),\xi) \dd\Pbb(\omega) \dd\nu_B(\xi)
  &=\int c_B(g_1\dotsm g_n, \xi) \dd\mu(g_1) \dotsm \dd\mu(g_n) \dd\nu_B(\xi)
  \\&
  =\sum_{k=1}^n \int c_B(g_k, g_{k+1}\dotsm g_n \xi)  \dd\mu(g_k)\dotsm \dd\mu(g_n)\dd\nu_B(\xi).
  \end{align*}
Since the measure $\nu_B$ is stationary, the point $g_{k+1}\dotsm g_n
\xi$ is distributed according to $\nu_B$. Hence, the terms in the
above sum do not depend on $k$. We get
  \begin{equation}
  \label{eq:equal_drift}
  \int_{\Gamma\times\partial_B \Gamma} c_B(g,\xi)\dd\mu(g) \dd\nu_B(\xi)
  = \frac{1}{n} \int c_B(X_n(\omega),\xi) \dd\Pbb(\omega) \dd\nu_B(\xi).
  \end{equation}
We have $\abs{c_B(X_n,\xi)}/n \leq \abs{X_n}/n$, which converges in
$L^1$ and almost surely to $\ell$. Hence, the sequence of functions
$c_B(X_n(\omega),\xi)/n$ is uniformly integrable on $\Omega\times
\partial_B \Gamma$. Moreover, $X_n$ converges almost surely to a
point on the boundary $\partial\Gamma$, distributed according to the
exit measure, which has no atom. It follows that, for all $\xi$, the
trajectory $X_n(\omega)$ converges almost surely to a point different
from $\pi_B(\xi)$. This implies that, almost surely, one has
$c_B(X_n, \xi) = \abs{X_n}+O(1)$, giving in particular $c_B(X_n,\xi)/
n\to \ell$. The result follows by taking the limit in $n$ in the
equality~\eqref{eq:equal_drift}.
\end{proof}

This formula easily implies that the drift depends continuously on
the measure, as explained in~\cite{erschler_kaim}.

\begin{prop}
\label{prop:ell_continuous}
Let $(\Gamma,d)$ be a metric hyperbolic group. Consider a sequence of
probability measures $\mu_i$ with finite first moment, converging
simply to a nonelementary probability measure $\mu$ (i.e.,
$\mu_i(g)\to \mu(g)$ for all $g\in \Gamma$). Assume moreover that
$L(\mu_i)\to L(\mu)$. Then $\ell(\mu_i) \to \ell(\mu)$.
\end{prop}
\begin{proof}
Let $\nu_i$ be stationary measures for $\mu_i$ on $\partial_B\Gamma$.
Taking a subsequence if necessary, we may assume that $\nu_i$
converges to a limiting measure $\nu$. By continuity of the action on
the boundary, it is stationary for $\mu$.

For each $g\in \Gamma$, the quantity $\int_{\partial_B \Gamma}
c_B(g,\xi) \dd\nu_i(\xi)$ converges to $\int_{\partial_B \Gamma}
c_B(g,\xi) \dd\nu(\xi)$ since $\xi\mapsto c_B(g,\xi)$ is continuous.
Averaging over $g$ (and using the assumption $L(\mu_i)\to L(\mu)$ to
get a uniform domination), we deduce that
  \begin{equation*}
  \sum_{g\in \Gamma} \mu_i(g) \int_{\partial_B \Gamma} c_B(g,\xi) \dd\nu_i(\xi)
  \to
  \sum_{g\in \Gamma} \mu(g) \int_{\partial_B \Gamma} c_B(g,\xi) \dd\nu(\xi).
  \end{equation*}
Together with the formula~\eqref{eq:asymp_ell} for the drift, this
completes the proof.
\end{proof}

In this proposition, it is important that $\mu$ is nonelementary: the result
is wrong otherwise. For instance, in the infinite dihedral group $\Z \rtimes
\Z/2$, the measures $\mu_i = (1-2^{-i}) \delta_{(1,0)} + 2^{-i}
\delta_{(0,1)}$ have zero drift since the $\Z/2$ element symmetrizes
everything in $\Z$, while the limiting measure $\mu=\delta_{(1,0)}$ has drift
$1$. The reason is the non-uniqueness of the stationary measure for $\mu$ on
the boundary.

\subsection{The entropy}

Let $\Gamma$ be a countable group. Consider a probability measure $\mu$ on
$\Gamma$, with finite time one entropy $H(\mu)$ (defined in~\eqref{defLmu}).
The entropy of the random walk has been defined in~\eqref{eq:def_h_ell} as
$h(\mu)=\lim H(\mu^{*n})/n$. Let $X_n=g_1\dotsm g_n$ be the position at time
$n$ of the random walk generated by $\mu$ (where the $g_i$ are independent
and distributed according to $\mu$). Then, almost surely, $h(\mu) = \lim
(-\log \mu^{*n}(X_n))/n$. The fundamental inequality~\eqref{eq:fundam} shows
that if $h>0$ then $\ell>0$.

The entropy has several equivalent characterizations. The first one
is in terms of the size of the typical support of the random walk:
This support has size roughly $e^{hn}$. The following lemma follows
from~\cite[Proposition~1.13]{haissinsky}.
\begin{lem}
\label{lem:charac_entropy} Consider a probability measure $\mu$ with
$H(\mu)<\infty$ on a countable group. Let $h=h(\mu)$ be its asymptotic
entropy. Let $\eta>0$ and $\epsilon>0$.
\begin{enumerate}
\item For large enough $n$, there exists a subset $K_n$ of $\Gamma$ with
    $\mu^{*n}(K_n) \geq 1-\eta$ and $\Card{K_n} \leq e^{(h+\epsilon)n}$.
\item For large enough $n$, there exists no subset $K_n$ of $\Gamma$ with
    $\mu^{*n}(K_n) \geq \eta$ and $\Card{K_n} \leq e^{(h-\epsilon)n}$.
\end{enumerate}
\end{lem}

Another description is in terms of the Poisson boundary of the walk.
To avoid general definitions, let us only state this description for
measures on hyperbolic groups. The following proposition is a
consequence of~\cite{kaimanovich_poisson}.

\begin{prop}
\label{prop:express_entropy}
Let $\Gamma$ be a hyperbolic group. Let $\mu$ be a nonelementary
probability measure on $\Gamma$ with $H(\mu)<\infty$. Let $\nu$ be
its unique stationary measure on $\partial \Gamma$. Define the Martin
cocycle on $\Gamma \times \partial\Gamma$ by $c_M(g,\xi)=-\log (\dd
g^{-1}_* \nu / \dd\nu)(\xi)$. Then
  \begin{equation}
  \label{eq:asymp_h}
  h(\mu) \geq \int_{\Gamma \times \partial \Gamma} c_M(g,\xi) \dd\mu(g) \dd\nu(\xi),
  \end{equation}
with equality if $\mu$ has a logarithmic moment.
\end{prop}
When $\mu$ has a logarithmic moment, this proposition has a very similar
flavor to Proposition~\ref{prop:express_drift} expressing the drift of a
random walk. Indeed, for symmetric measures, \cite{BHM_2} interprets
Proposition~\ref{prop:express_entropy} as a special case of
Proposition~\ref{prop:express_drift}, for a distance $d=d_\mu$ related to the
random walk, the Green distance, which we defined in Theorem~\ref{thm:BHM}.
This distance is hyperbolic if $\mu$ is admissible and has a superexponential
moment, by~\cite{ancona, gouezel_infinite_support}. It is not geodesic in
general, but this is not an issue since we were careful enough to state
Proposition~\ref{prop:express_drift} without this assumption. The Busemann
cocycle for the Green distance is precisely the Martin cocycle.

An important difference between the formulas~\eqref{eq:asymp_ell} for the
drift and~\eqref{eq:asymp_h} for the entropy is that, in the latter
situation, the cocycle $c_M$ depends on the measure $\nu$ (and, therefore, on
$\mu$). This makes it more complicated to prove continuity statements such as
Proposition~\ref{prop:ell_continuous} for the entropy. Nevertheless, Erschler
and Kaimanovich proved in~\cite{erschler_kaim} that, in hyperbolic groups,
the entropy also depends continuously on the measure. As $h(\mu) = \inf
H(\mu^{*n})/n$ by subadditivity, it is easy to prove that when $\mu_i \to
\mu$ one has $\limsup h(\mu_i) \leq h(\mu)$. The main difficulty to prove the
continuity is to get lower bounds. We will need a slightly stronger (and more
pedestrian) version of the results of~\cite{erschler_kaim} to prove
Theorem~\ref{thm:tends_to_v}. Although our argument may seem very different
at first sight from the arguments in~\cite{erschler_kaim}, the techniques are
in fact closely related (an illustration is that we can recover with our
techniques the result of Kaimanovich that, for measures with finite
logarithmic moment, equality holds in~\eqref{eq:asymp_h}, i.e., the Poisson
boundary coincides with the geometric boundary, see
Remark~\ref{rmk:recover_kaim}). Our main criterion to get lower bounds on the
entropy is the following. We write $\Sbb^k=\{g\in \Gamma \st \abs{g}\in (k-1,
k]\}$ for the thickened sphere, so that the union of these spheres covers the
whole group.

\begin{thm}
\label{thm:h_grande_limite}
Let $(\Gamma,d)$ be a metric hyperbolic group. Let $\mu_i$ be a
sequence of nonelementary probability measures on $\Gamma$ with
$H(\mu_i)<\infty$. Let $\nu_i$ be the unique stationary measure for
$\mu_i$ on $\partial \Gamma$. Assume that:
\begin{enumerate}
\item The limit points of $\nu_i$ have no atom.
\item The sequence
  \begin{equation}
  \label{eq:def_hn}
  h_i = \adjustlimits\sum_k \sum_{g\in \Sbb^k} \mu_i(g) (-\log(\mu_i(g)/\mu_i(\Sbb^k)))
  \end{equation}
  tends to infinity.
\end{enumerate}
Then $\liminf h(\mu_i)/h_i \geq 1$.
\end{thm}
The quantity $h_i$ can be written
  \begin{equation*}
  h_i = \sum_{g\in \Gamma} \mu_i(g) (-\log \mu_i(g)) - \sum_k \mu_i(\Sbb^k) (-\log \mu_i(\Sbb^k)).
  \end{equation*}
The first term is the time one entropy $H(\mu_i)$ of the measure $\mu_i$. In
most reasonable cases, the second term is negligible. The theorem then states
that the asymptotic entropy $h(\mu_i)$ is comparable to the time one entropy
$H(\mu_i)$. In other words, if the measure is supported close to infinity,
and sufficiently spread out in the group (this is the meaning of the
assumption that the limit points of $\nu_i$ have no atom), then there are few
coincidences and the entropy does not decrease significantly with time.

To prove this theorem, we will use the following technical lemma.
\begin{lem}
\label{lem:minore_entropie}
On a probability space $(X,\mu)$, consider a nonnegative function $f$ with
average $1$. For any subset $A$ of $X$,
  \begin{equation*}
  \int_X (-\log f) \geq \mu(A) \paren*{-\log \int_A f} - 2 e^{-1}.
  \end{equation*}
\end{lem}
\begin{proof}
As the function $x\mapsto -\log x$ is convex, Jensen's inequality
gives $\int (-\log f) \geq -\log (\int f)$. The last quantity
vanishes when $\int f=1$.

Let $B\subset X$. Write $a=\int_B f \dd\mu/\mu(B)$. The measure
$\dd\mu/\mu(B)$ is a probability measure on $B$, and the function
$f/a$ has integral $1$ for this measure. The previous inequality
gives $\int_B (-\log(f/a)) \dd\mu/\mu(B) \geq 0$, that is,
  \begin{equation*}
  \int_B (-\log f) \dd\mu \geq -\mu(B)\log a = -\mu(B)\log\paren*{\int_B f}+\mu(B)\log \mu(B).
  \end{equation*}
The quantity $\mu(B)\log \mu(B)$ is bounded from below by
$\inf_{[0,1]} x \log x=-e^{-1}$. Therefore,
  \begin{equation*}
  \int_B (-\log f) \dd\mu \geq -\mu(B) \log\paren*{\int_B f} - e^{-1}.
  \end{equation*}

We apply this inequality to the complement $A^c$ of $A$. As $-\log
\paren*{\int_{A^c} f} \geq 0$, we get a lower bound $-e^{-1}$. Let us also
apply this inequality to $A$, and add the results. We obtain
  \begin{equation*}
  \int_X (-\log f)\dd\mu \geq -\mu(A) \log\paren*{\int_A f} -2e^{-1}.
  \qedhere
  \end{equation*}
\end{proof}

We will use the notion of shadow, due to Sullivan and considered in this
context by Coornaert~\cite{coornaert}. Let $C>0$ be large enough. The shadow
$\boO(g,C)$ of $g\in \Gamma$ is $\{\xi \in \partial \Gamma \st
\gromprod{g}{\xi}{e} \geq \abs{g}-C\}$. In geometric terms (and assuming the
space is geodesic), this is essentially the trace at infinity of geodesics
originating from $e$ and going through the ball $B(g, C)$. We will use the
following properties of shadows~\cite{coornaert}:
\begin{enumerate}
\item Their covering number is finite. More precisely, there
    exists $D>0$ (depending on $C$) such that, for any integer
    $k$, for any $\xi \in \partial \Gamma$,
    \begin{equation*}
    \Card{ \{g\in \Sbb^k \st \xi \in \boO(g,C)\} } \leq D.
    \end{equation*}
\item The preimages of shadows are large. More precisely, for any $\eta>0$,
    there exists $C>0$ such that, for all $g\in \Gamma$, the complement of
    $g^{-1} \boO(g,C)$ has diameter at most $\eta$ (for a fixed visual
    distance on the boundary).
\end{enumerate}

\begin{proof}[Proof of Theorem~\ref{thm:h_grande_limite}]
Fix $\epsilon>0$. As the limit points of $\nu_i$ have no atom, there exists
$\eta>0$ such that any ball of radius $\eta$ in $\partial\Gamma$ has measure
at most $\epsilon$ for $\nu_i$, for $i$ large enough. We can then choose a
shadow size $C$ so that $g^{-1} \boO(g,C)$ has for all $g$ a complement with
diameter at most $\eta$. This yields $\nu_i(g^{-1} \boO(g,C))\geq
1-\epsilon$.

By~\eqref{eq:asymp_h}, the entropy of $\mu_i$ satisfies
  \begin{equation*}
  h(\mu_i) \geq \sum_{g\in \Gamma}\mu_i(g) \int_{\partial \Gamma}
    \paren*{-\log \frac{\dd g^{-1}_* \nu_i}{\dd\nu_i}(\xi) } \dd\nu_i(\xi).
  \end{equation*}

The function $f_{i,g}=\frac{\dd g^{-1}_* \nu_i}{\dd\nu_i}(\xi)$ is
nonnegative and has integral $1$. For any $A\subset \partial\Gamma$,
Lemma~\ref{lem:minore_entropie} gives
  \begin{align*}
  \int_{\partial \Gamma} \paren*{-\log \frac{\dd g^{-1}_* \nu_i}{\dd\nu_i}(\xi) } \dd\nu_i(\xi)
  &\geq -\nu_i(A) \log\paren*{\int_A \frac{\dd g^{-1}_* \nu_i}{\dd\nu_i}(\xi) \dd\nu_i(\xi)} - 2e^{-1}
  \\&=-\nu_i(A) \log (g^{-1}_*\nu_i(A)) - 2e^{-1}
  \\&
  =-\nu_i(A) \log(\nu_i(g A))-2e^{-1}.
  \end{align*}
Let us take $A=g^{-1} \boO(g ,C)$, so that $\nu_i(A) \geq
1-\epsilon$. Summing over $g$, we get
  \begin{equation}
  \label{eq:qsoidufpioqsdf}
  h(\mu_i) \geq (1-\epsilon) \sum_{g\in \Gamma} \mu_i(g) (-\log \nu_i(\boO(g,C))) - 2e^{-1}.
  \end{equation}

We split the sum according to the spheres $\Sbb^k$. Let $\Sigma_k =
\sum_{g\in \Sbb^k} \nu_i(\boO(g,C))$, it is at most $D$ since the
shadows have a covering number bounded by $D$. We have
  \begin{multline*}
  \sum_{g\in \Sbb^k} \mu_i(g) (-\log \nu_i(\boO(g,C)))
  \\
  =-\mu_i(\Sbb^k) \sum_{g\in \Sbb^k} \frac{\mu_i(g)}{\mu_i(\Sbb^k)}
    \left[ \log \paren*{ \frac{ \nu_i(\boO(g,C))}{\Sigma_k \mu_i(g)/\mu_i(\Sbb^k)}}
    + \log \Sigma_k + \log(\mu_i(g)/\mu_i(\Sbb^k)) \right].
  \end{multline*}
The point of this decomposition is that the function on $\Sbb^k$
given by $\phi : g\mapsto \frac{ \nu_i(\boO(g,C))}{\Sigma_k
\mu_i(g)/\mu_i(\Sbb^k)}$ has integral $1$ for the probability measure
$\mu_i(g)/\mu_i(\Sbb^k)$. By Jensen's inequality, the integral of
$-\log \phi$ is nonnegative. This yields
  \begin{equation*}
  \sum_{g\in \Sbb^k} \mu_i(g) (-\log \nu_i(\boO(g,C)))
  \geq -\mu_i(\Sbb^k) \log D + \sum_{g\in \Sbb^k} \mu_i(g) (-\log(\mu_i(g)/\mu_i(\Sbb^k)).
  \end{equation*}
Summing over $k$, we deduce from~\eqref{eq:qsoidufpioqsdf} the
inequality
  \begin{equation*}
  h(\mu_i) \geq (1-\epsilon)h_i -2e^{-1} -\log D.
  \end{equation*}
As $h_i$ tends to infinity, this gives $h(\mu_i) \geq
(1-2\epsilon)h_i$ for large enough $i$, completing the proof.
\end{proof}

To apply the previous theorem, we need to estimate $h_i$. In this
respect, the following lemma is often useful.
\begin{lem}
\label{lem:eq_borne_inf_hn}
Let $R_i\geq 1$. The quantity $h_i$ defined in~\eqref{eq:def_hn}
satisfies
  \begin{equation*}
  h_i \geq  \sum_{\abs{g}\leq R_i} \mu_i(g) (-\log \mu_i(g)) - \log(2+R_i).
  \end{equation*}
\end{lem}
\begin{proof}
In the definition of $h_i$, all the terms are nonnegative.
Restricting the sum to those $g$ with $\abs{g}\leq R_i$, we get
  \begin{align*}
  h_i &
  \geq \adjustlimits\sum_{k\leq R_i} \sum_{g\in \Sbb^k} \mu_i(g) (-\log (\mu_i(g)/\mu_i(\Sbb^k)))
  \\&
  = \sum_{\abs{g} \leq R_i} \mu_i(g) (-\log \mu_i(g)) - \sum_{k\leq R_i} \mu_i(\Sbb^k) (-\log \mu_i(\Sbb^k)).
  \end{align*}
A probability measure supported on a set with $N$ elements has
entropy at most $\log N$. The number $\mu_i(\Sbb^k)$ for $0\leq k
\leq R_i$ are not a probability measure in general, let us add a last
atom with mass $m=\mu_i(\bigcup_{k>R_i} \Sbb^k)$. We are considering
a space of cardinality $R_n+2$, hence
  \begin{equation*}
  m(-\log m) + \sum_{k\leq R_i} \mu_i(\Sbb^k) (-\log \mu_i(\Sbb^k)) \leq \log(2+R_i),
  \end{equation*}
completing the proof.
\end{proof}

Let us see how Theorem~\ref{thm:h_grande_limite} implies a slightly
stronger version of the continuity result for the entropy of Erschler
and Kaimanovich~\cite{erschler_kaim}.
\begin{thm}
\label{thm:h_continuous}
Let $\Gamma$ be a hyperbolic group. Consider a probability measure
$\mu$ with finite time one entropy and finite logarithmic moment. Let
$\mu_i$ be a sequence of probability measures converging simply to
$\mu$ with $H(\mu_i) \to H(\mu)$. Then $h(\mu_i) \to h(\mu)$.
\end{thm}
The assumption $H(\mu_i) \to H(\mu)$ ensures that there is no
additional entropy in $\mu_i$ coming from neighborhoods of infinity
that would disappear in the limit. It is automatic if the support of
$\mu_i$ is uniformly bounded or if $\mu_i$ satisfies a uniform $L^1$
domination, but it is much weaker. For instance, it is allowed that
the $\mu_i$ have no finite logarithmic moment.

The main lemma for the proof is a lower bound on the entropy,
following from Theorem~\ref{thm:h_grande_limite}.

\begin{lem}
\label{lem:liminf_hn_grand}
Let $\Gamma$ be a hyperbolic group. Consider a probability measure
$\mu$ with finite time one entropy and finite logarithmic moment. Let
$\mu_i$ be a sequence of measures converging simply to $\mu$. Then
$\liminf h(\mu_i) \geq h(\mu)$.
\end{lem}
\begin{proof}
Since the result is trivial if $h(\mu)=0$, we can assume that
$h(\mu)>0$.

Let $\epsilon>0$. For large $n$, most atoms for $\mu^{*n}$ have a
probability at most $e^{-(1-\epsilon) n h(\mu)}$. Moreover, since
$\mu$ has a finite logarithmic moment, $\log \abs{X_n}/n$ tends
almost surely to $0$ by~\cite[Proposition 2.3.1]{aaronson_book}.
Therefore, the set
  \begin{equation*}
  K_n=\{g \st \mu^{*n}(g) \leq e^{-(1-\epsilon)n h(\mu)},\ \abs{g} \leq e^{\epsilon n}\}
  \end{equation*}
has measure tending to $1$. In particular $\mu^{*n}(K_n) \geq
1-\epsilon$ for large $n$. We get
  \begin{align*}
  \sum_{\abs{g} \leq e^{\epsilon n}} \mu^{*n}(g)(-\log \mu^{*n}(g))
  &\geq \sum_{g\in K_n} \mu^{*n}(g)(-\log \mu^{*n}(g))
  \geq \sum_{g\in K_n} \mu^{*n}(g)(1-\epsilon) n h(\mu)
  \\&
  = \mu^{*n}(K_n) (1-\epsilon) n h(\mu)
  \geq (1-\epsilon)^2 n h(\mu).
  \end{align*}
For each fixed $n$, the measures $\mu_i^{*n}$ converge to $\mu^{*n}$
when $i$ tends to infinity. Hence, we get for large enough $i$ the
inequality
  \begin{equation*}
  \sum_{\abs{g} \leq e^{\epsilon n}} \mu_i^{*n}(g)(-\log \mu_i^{*n}(g)) \geq (1-\epsilon)^3 n h(\mu).
  \end{equation*}

Letting $\epsilon$ tend to $0$ (and, therefore, $n$ to infinity), we
deduce the existence of sequences  $n_i \to \infty$ and
$\epsilon_i\to 0$ such that, for any $i$,
  \begin{equation*}
  \sum_{\abs{g} \leq e^{\epsilon_i n_i}} \mu_i^{*n_i}(g)(-\log \mu_i^{*n_i}(g))
  \geq (1-\epsilon_i)^3 n_i h(\mu).
  \end{equation*}
Let $\tilde\mu_i = \mu_i^{*n_i}$. Its stationary measure $\nu_i$ is
also the stationary measure of $\mu_i$, by uniqueness. Any limit
point of $\nu_i$ is stationary for $\mu$, and is therefore atomless
since $\mu$ is nonelementary as $h(\mu)>0$. The assumptions of
Theorem~\ref{thm:h_grande_limite} are satisfied by the sequence
$\tilde\mu_i$. Moreover, Lemma~\ref{lem:eq_borne_inf_hn} yields
  \begin{equation*}
  h_i \geq (1-\epsilon_i)^3 n_i h(\mu) - 2\epsilon_i n_i \geq (1-C \epsilon_i) n_i h(\mu).
  \end{equation*}
Theorem~\ref{thm:h_grande_limite} ensures that $\liminf
h(\tilde\mu_i)/h_i\geq 1$. As $h(\tilde\mu_i) = n_i h(\mu_i)$, this
gives $\liminf h(\mu_i)\geq h(\mu)$ as desired.
\end{proof}

\begin{proof}[Proof of Theorem~\ref{thm:h_continuous}]
For fixed $n$, the sequence $\mu_i^{*n}$ converges simply to
$\mu^{*n}$. Moreover, $H(\mu_i^{*n})\to H(\mu^{*n})$ since there is
no loss of entropy at infinity by assumption. Choose $n$ such that
$H(\mu^{*n}) \leq n(1+\epsilon) h(\mu)$. We get $H(\mu_i^{*n}) /n
\leq (1+2\epsilon) h(\mu)$ for large enough $i$. As $h(\mu_i) \leq
H(\mu_i^{*n}) /n$, this shows that $\limsup h(\mu_i) \leq h(\mu)$
(this is the classical semi-continuity property of entropy, valid in
any group).

For the reverse inequality $\liminf h(\mu_i) \geq h(\mu)$, we apply
Lemma~\ref{lem:liminf_hn_grand}.
\end{proof}

\begin{rmk}
\label{rmk:recover_kaim} Let $h(\mu, \partial \Gamma) = \int_{\Gamma\times
\partial\Gamma} (-\log \dd g^{-1}_* \nu/\dd\nu)(\xi) \dd\mu(g) \dd\nu(\xi)$ where $\nu$ is the
stationary measure for $\mu$ on $\partial \Gamma$. In general, $h(\mu)\geq
h(\mu, \partial \Gamma)$ with equality if and only if $(\partial\Gamma,\nu)$
is the Poisson boundary of $(\Gamma,\mu)$. A theorem of
Kaimanovich~\cite{kaimanovich_poisson} asserts that, when $\mu$ has finite
entropy and finite logarithmic moment, $h(\mu,
\partial \Gamma)=h(\mu)$. We can recover this theorem using the previous
arguments. Indeed, what the proof of Theorem~\ref{thm:h_grande_limite} really
shows is that $\liminf h(\mu_i,
\partial \Gamma)/h_i \geq 1$. Hence, Lemma~\ref{lem:liminf_hn_grand} proves
that $\liminf h(\mu_i,
\partial \Gamma) \geq h(\mu)$ if $\mu_i$ converges simply to a
measure $\mu$ with a logarithmic moment. Taking $\mu_i=\mu$ for all
$i$, we obtain in particular $h(\mu, \partial \Gamma) \geq h(\mu)$,
as desired.
\end{rmk}

\subsection{A criterion to bound the entropy from below}
\label{subsec:tends_to_v}

In order to prove Theorem~\ref{thm:tends_to_v} on the entropy of the
uniform measure on balls, we want to apply
Theorem~\ref{thm:h_grande_limite}. Thus, we need a criterion to check
that limit points of stationary measures have no atom.

\begin{lem}
\label{lem:describe_nu}
Let $\Gamma$ be a hyperbolic group. Let $\mu_i$ be a sequence of
probability measures on $\Gamma$. Assume that, on the space
$\Gamma\cup \partial \Gamma$, the sequence $\mu_i$ converges to a
limit $\nu$ which is supported on $\partial\Gamma$. Assume moreover
that the limit points of $\check\mu_i$ (defined by
$\check\mu_i(g)=\mu_i(g^{-1})$) have no atom. Then the stationary
measures $\nu_i$ associated to $\mu_i$ also converge to $\nu$.
\end{lem}
\begin{proof}
We fix a word distance $d$ on $\Gamma$. Let $f$ be a continuous function on
$\Gamma \cup \partial \Gamma$. Let us show that, uniformly in $\xi\in
\partial\Gamma$, the integral $\int f(g \xi) \dd\mu_i(g)$ is close to $\int f(g)
\dd\mu_i(g)$. We estimate the difference as
  \begin{align*}
  \abs*{\int (f(g \xi)-f(g)) \dd\mu_i(g) }
  \leq {} &
  \int \abs{ f(g\xi)-f(g)} 1( \gromprod{g\xi}{g}{e} > C)\dd\mu_i(g)
  \\&  + 2 \norm{f}_\infty \int 1(\gromprod{g\xi}{g}{e} \leq C)\dd\mu_i(g),
  \end{align*}
where $C$ is a fixed constant. If $C$ is large enough,
$\abs{f(x)-f(y)} \leq \epsilon$ when $\gromprod{x}{y}{e} > C$, by
uniform continuity of $f$. Hence, the first integral is bounded by
$\epsilon$. For the second integral, we use the formula
$\gromprod{gx}{g}{e}=\abs{g}-\gromprod{x}{g^{-1}}{e}$, valid for any
$x\in \Gamma$ (it follows readily from the
definition~\eqref{def:gromov_product} of the Gromov product). This
equality does not extend to the boundary since the Gromov product
there is only well defined up to an additive constant $D$.
Nevertheless, we get $\gromprod{g\xi}{g}{e} \geq
\abs{g}-\gromprod{\xi}{g^{-1}}{e}- D$. Hence, the second integral is
bounded by
  \begin{equation}
  \label{eq:oiwuvcpoiuwxvc}
  \mu_i\{ g\st \abs{g}-C-D \leq \gromprod{\xi}{g^{-1}}{e}\}.
  \end{equation}
If $\abs{g}$ is large, the points $g$ with $\gromprod{\xi}{g^{-1}}{e}
\geq \abs{g}-C-D$ are such that $g^{-1}$ belongs to a small
neighborhood of $\xi$ in $\Gamma\cup \partial \Gamma$. As the limit
points of $\check\mu_i$ are supported on $\partial \Gamma$ and have
no atom, it follows that~\eqref{eq:oiwuvcpoiuwxvc} converges to $0$
when $i$ tends to infinity, uniformly in $\xi$.

We have proved that
  \begin{equation*}
  \sup_{\xi \in \partial\Gamma} \abs*{\int f(g\xi) \dd\mu_i(g) - \int f(g) \dd\mu_i(g)} \to 0.
  \end{equation*}
By stationarity,
  \begin{equation*}
  \int_{\xi\in\partial\Gamma} f(\xi) \dd\nu_i(\xi)
  = \int_{\xi\in \partial\Gamma} \paren*{\int f(g\xi) \dd\mu_i(g)} \dd\nu_i(\xi).
  \end{equation*}
Combining these equations, we get $\int f(\xi)\dd\nu_i(\xi) - \int
f(g) \dd\mu_i(g)\to 0$. This shows that the limit points of $\nu_i$
and $\mu_i$ are the same.
\end{proof}

Let us now consider the uniform measure $\mu_i$ on the ball of radius
$i$, as in Theorem~\ref{thm:tends_to_v}. The next lemma follows from
the techniques of~\cite{coornaert}.
\begin{lem}
\label{lem:converge_PS}
Let $(\Gamma,d)$ be a metric hyperbolic group. Let $\rho_i$ be the
uniform measure on the ball of radius $i$. Let $\rho_\infty$ be the
Patterson-Sullivan of $(\Gamma,d)$ constructed in~\cite{coornaert}
(it is supported on $\partial\Gamma$ and atomless). Then the limit
points of $\rho_i$ are equivalent to $\rho_\infty$, with a density
bounded from above and from below.
\end{lem}
\begin{proof}
Let $C$ be large enough. We will use the shadows $\boO(g,C)$ as
defined before the proof of Theorem~\ref{thm:h_grande_limite}. The
main property of $\rho_\infty$ is that it satisfies
  \begin{equation}
  \label{eq:asymp_PS}
  K_0^{-1} e^{-v \abs{g}} \leq \rho_\infty(\boO(g,C)) \leq K_0 e^{-v\abs{g}},
  \end{equation}
where $K_0$ is a constant only depending on $C$ and $v$ is the growth
of $(\Gamma,d)$ (Proposition~6.1 in~\cite{coornaert}).

Let $\mu_i$ be the uniform measure on thickened spheres $S_i=\{g \st i\leq
\abs{g}\leq i+L\}$, where $L$ is large enough so that the cardinality of
$S_i$ grows like $e^{iv}$, see the proof of Theorem 7.2 in~\cite{coornaert}.
Let us push $\mu_i$ to a measure $\tilde \mu_i$ on $\partial\Gamma$, by
choosing for each $g \in S_i$ a corresponding point in its shadow. It is
clear that $\mu_i$ and $\tilde\mu_i$ have the same limit points, since the
diameter of the shadows tends uniformly to $0$ when $i\to\infty$. We will
prove that the limit points of $\tilde\mu_i$ are equivalent to $\rho_\infty$.
The same result follows for $\mu_i$ and then $\rho_i$.

The shadows of $g \in S_i$ have a covering number which is bounded from above
by a constant $D$, and from below by $1$ if $C$ is large enough. Hence, the
measures $\tilde \mu_i$ satisfy
  \begin{equation*}
  K_1^{-1} e^{-iv} \leq \tilde\mu_i(\boO(g,C)) \leq K_1 e^{-iv},
  \end{equation*}
for any $g\in S_i$. This is comparable to $\rho_\infty(\boO(g,C))$
by~\eqref{eq:asymp_PS}, up to a multiplicative constant $K_2$. Consider a
limit $\tilde\mu$ of a sequence $\tilde\mu_{i_n}$, let us prove that it is
uniformly equivalent to $\rho_\infty$. We will only prove that $\tilde\mu
\leq DK_2 \rho_\infty$, the other inequality is proved in the same way. By
regularity of the measures, it suffices to check this inequality on compact
sets.

Let $A$ be a compact subset of $\partial \Gamma$, and $\epsilon>0$. By
regularity of the measure $\rho_\infty$, there is an open neighborhood $U$ of
$A$ with $\rho_\infty(U) \leq \rho_\infty(A)+\epsilon$. Consider $B$ a
compact neighborhood of $A$, included in $U$, with $\tilde\mu(\partial B)=0$
(such a set exists, since among the sets $B_r=\{\xi \st d(\xi, A)\leq r\}$,
at most countably of them many have a boundary with nonzero measure). For
large enough $i$, the shadows $\boO(g,C)$ with $g\in S_i$ which intersect $B$
are contained in $U$. Therefore,
\begin{equation*}
  \tilde \mu_i(B) \leq \sum_{g\in S_i, \boO(g,C)\cap B\not=\emptyset} \tilde \mu_i(\boO(g,C))
  \leq K_2 \sum_{g\in S_i, \boO(g,C)\cap B\not=\emptyset} \rho_\infty(\boO(g,C))
  \leq D K_2 \rho_\infty(U).
\end{equation*}
As $\tilde\mu(\partial B)=0$, the sequence $\tilde \mu_{i_n}(B)$ tends to
$\tilde\mu(B)$. We obtain $\tilde\mu(B) \leq D K_2 \rho_\infty(U)$. As $A$ is
included in $B$, we get $\tilde\mu(A) \leq DK_2 (\rho_\infty(A)+\epsilon)$.
Letting $\epsilon$ tend to $0$, this gives $\tilde\mu(A) \leq D K_2
\rho_\infty(A)$, as desired.
\end{proof}

\begin{proof}[Proof of Theorem~\ref{thm:tends_to_v}]
Let $\rho_i$ be the uniform measure on the ball of radius $i$ (which has
cardinality in $[C^{-1}e^{iv}, Ce^{iv}]$). We wish to apply
Theorem~\ref{thm:h_grande_limite} to this sequence of measures. First, by
Lemmas~\ref{lem:describe_nu} and~\ref{lem:converge_PS}, the limit points of
the stationary measures $\nu_i$ are equivalent to the Patterson-Sullivan
measure. Therefore, they have no atom. Second,
Lemma~\ref{lem:eq_borne_inf_hn} shows that the quantity $h_i$
in~\eqref{eq:def_hn} satisfies $h_i \geq iv - \log C -\log(2+i)$. This tends
to infinity. Hence, Theorem~\ref{thm:h_grande_limite} applies, and gives
$h(\rho_i) \geq (1-\epsilon)iv$ for large $i$.

Using the fundamental inequality $h\leq \ell v$ and the trivial bound
$\ell(\rho_i) \leq L(\rho_i) \leq i$, we get
  \begin{equation*}
  (1-\epsilon) iv \leq h(\rho_i) \leq \ell(\rho_i) v \leq iv.
  \end{equation*}
It follows that $h(\rho_i) \sim iv$ and $\ell(\rho_i) \sim i$.
\end{proof}

\begin{rmk}
Our technique also applies to estimate the entropy of other measures,
for instance the measure $\mu_s=\sum e^{-s\abs{g}}\delta_g / \sum
e^{-s\abs{g}}$ classically used in the construction of the
Patterson-Sullivan measure. Indeed, $\mu_s$ converges when $s\searrow
v$ to $\rho_\infty$, which has no atom. Moreover, writing $Z_s = \sum
e^{-s\abs{g}}$, we have $H(\mu_s)=sL(\mu_s) + \log Z_s$. One checks
that $\log Z_s$ is negligible with respect to $H(\mu_s)$, and that
the quantity $h_s$ from~\eqref{eq:def_hn} is also equivalent to
$H(\mu_s)$. Hence, Theorem~\ref{thm:h_grande_limite} gives
  \begin{equation*}
  H(\mu_s) (1+o(1)) \leq h_s(1+o(1)) \leq h(\mu_s) \leq \ell(\mu_s) v \leq L(\mu_s) v \leq H(\mu_s)(1+o(1)).
  \end{equation*}
These inequalities show that $h(\mu_s)/\ell(\mu_s) \to v$.
\end{rmk}

\begin{rmk}
One could imagine another strategy to find finitely supported
measures $\mu_i$ for which $h(\mu_i)/\ell(\mu_i)\to v$. First, find a
nice measure $\mu$ for which the stationary measure $\nu$ at infinity
is precisely the Patterson-Sullivan measure (which implies that
$h(\mu)=\ell(\mu) v$ since the Martin cocycle and the Busemann
cocycle coincide). Let $\mu_i$ be a truncation of $\mu$. Since it
converges to $\mu$, the continuity results for the drift and the
entropy imply that $h(\mu_i)/\ell(\mu_i) \to h(\mu)/\ell(\mu)=v$.

We were not able to implement successfully this strategy. Given a
measure $\nu$, there is a general technique due to Connell and
Muchnik~\cite{connell_muchnik} to get a measure $\mu$ on $\Gamma$
with $\mu*\nu=\nu$. This technique requires a continuity assumption
on $\xi \mapsto (\dd g_*\nu/\dd\nu)(\xi)$, which is not satisfied in
our setting for $\nu=\rho_\infty$. However, in nice groups such as
surface groups, this function is, for every $g$, continuous at all
but finitely many points. The technique of~\cite{connell_muchnik} can
be adapted to such a situation (in the proof of their Theorem~6.2,
one should just take sets $Y_n$ that avoid the discontinuities of the
spikes we have already used). Unfortunately, the resulting measure
$\mu$ (which satisfies $\mu*\nu=\nu$) has infinite moment and
infinite entropy, and is therefore useless for our purposes.
\end{rmk}

\section{Rigidity for admissible measures}

In this section, we prove Theorem~\ref{thm:hlv_indice_fini}. Assume that
$(\Gamma,d)$ is a hyperbolic group endowed with a word distance, which is not
virtually free. Let $\mu$ be a probability measure on $\Gamma$, with a
superexponential moment, such that $\Gamma_\mu^+$ is a finite index subgroup
of $\Gamma$. We want to prove that $h(\mu)<\ell(\mu)v$. We argue by
contradiction, assuming that $h(\mu)=\ell(\mu) v$. Assume first that
$\Gamma_\mu^+=\Gamma$.

Since we are assuming the equality $h(\mu)=\ell(\mu)v$, Theorem~\ref{thm:BHM}
implies that
\begin{equation*}
\abs{d_\mu(e,g)-v d(e,g)}\leq C.
\end{equation*}

As a warm-up, let us first deal with the baby case $C=0$. Then the distances
$d_\mu$ and $d$ are proportional, hence they define the same Busemann
boundary. The Busemann boundary $\partial_B\Gamma$ corresponding to $d$ is
totally discontinuous since the distance $d$ takes integer values (it is a
word distance). On the other hand, the Busemann boundary associated to the
Green metric $d_\mu$ is known as the Martin boundary of the random walk
$(\Gamma,\mu)$. By~\cite{ancona} and~\cite{gouezel_infinite_support}, it is
homeomorphic to the boundary $\partial\Gamma$ of $\Gamma$. Since the group
$\Gamma$ is not virtually free, its boundary $\partial\Gamma$ is not totally
discontinuous (see~\cite[Theorem 8.1]{kapovich_survey}), hence a
contradiction.

Let us now go back to the general situation, when $C$ is nonzero (but still
assuming $\Gamma_\mu^+=\Gamma$). The argument is more complicated, but it
still relies on the same facts: the boundary is not totally disconnected,
while the word distance is integer valued (we will not use directly this
fact, rather the fact that stable translation lengths are rational, see
Lemma~\ref{lem_rationnel}). These two opposite features will give rise to a
contradiction.

In order to get rid of the constant $C$, we will need an homogenized version
of the inequality $\abs{d_\mu(e,g)-v d(e,g)}\leq C$. This is
Lemma~\ref{lem_explicite} below. The homogenized quantity associated to the
distance $d$ is called the stable translation length. For an element $g$ of
$\Gamma$, it is defined by $l(g) = \lim\abs{g^n}/n$ (it exists by
subadditivity).

Recall that we write $c_M(g,\xi)$ for the Martin cocycle associated to the
random walk, defined in Proposition~\ref{prop:express_entropy}. It satisfies
the cocycle relation of Definition~\ref{def:cocycle}. We will not use its
probabilistic definition, but rather the fact that the Martin cocycle is the
Busemann cocycle associated to the Green distance $d_\mu$ of
Theorem~\ref{thm:BHM}. In other words, $c_M(g,\xi) = \lim_{x\to \xi}
d_\mu(g^{-1},x)-d_\mu(e,x)$ (and this limit exists).

\begin{lem}
\label{lem_explicite} For $g\in \Gamma$ with infinite order, $c_M(g, g^+) = v
l(g)$.
\end{lem}
\begin{proof}
Recall that we are assuming that the equality $h(\mu)=\ell(\mu)v$ holds,
therefore we have $\abs{d_\mu(e,g)-v d(e,g)}\leq C$. It follows that the
cocycle $c_M$ corresponding to $d_\mu$ and the cocycle $c_B$ corresponding to
the distance $d$ satisfy $\abs{c_M - v c_B} \leq 2C$. Note that $c_B$ is not
defined on the geometric boundary, but on the horoboundary, so the proper way
to write this inequality is $\abs{c_M(g, \pi_B(\xi)) - v c_B(g,\xi)} \leq 2C$
for any $g\in \Gamma$ and any $\xi \in
\partial_B \Gamma$.

Let $\xi \in \partial_B \Gamma$ with $\pi_B(\xi)\not=g^-$. Then $\lim
c_B(g^n,\xi)/n = \lim h_\xi(g^{-n})/n = l(g)$. We choose $\xi$ with
$\pi_B(\xi)=g^+$, to get
  \begin{equation*}
  \lim c_M(g^n, g^+)/n = \lim v c_B(g^n, \xi)/n \pm 2C/n = v l(g).
  \end{equation*}
As $g^+$ is $g$-invariant, the cocycle equation for $c_M$ on $\partial\Gamma$
gives $c_M(g, g^+) = c_M(g^n, g^+)/n$. This converges to $v l(g)$ when $n\to
\infty$ by the previous equation.
\end{proof}

The proof of Theorem~\ref{thm:hlv_indice_fini} uses the following general
result on cocycles.

\begin{prop}
\label{prop:cocycle} Let $\Gamma$ be a hyperbolic group which is not
virtually free. Let $c:\Gamma \times \partial \Gamma \to \R$ be a H\"older
cocycle, such that any hyperbolic element $g$ satisfies $c(g,g^+) \in \Z$.
Then there exists a hyperbolic element $g\in\Gamma$ with $c(g,g^-) =
c(g,g^+)$.
\end{prop}

Applied to the Busemann cocycle, this proposition implies that if a convex
cocompact negatively curved manifold has a fundamental group which is not
virtually free, then its length spectrum is not arithmetic, i.e., the lengths
of its closed geodesics generate a dense subgroup of $\R$. This result is
already known, see~\cite[Page~205]{dalbo}. It is proved in this article using
crossratios. This argument based on crossratios can be used to prove
Proposition~\ref{prop:cocycle} in full generality. However, we will give a
different, more direct, proof.

We will use the following topological lemma.
\begin{lem}
\label{lem:existe_axe} Let $g$ be a hyperbolic element in a hyperbolic group
$\Lambda$ with connected boundary. There exists an arc $I$ (i.e., a subset of
$\partial\Lambda$ homeomorphic to $[0,1]$) joining $g^-$ and $g^+$, invariant
under an iterate $g^i$ of $g$.
\end{lem}
\begin{proof}
We will use nontrivial results on the topology of $\partial \Lambda$. When it
is connected, then it is also locally connected by~\cite{swarup}. Hence, it
is also path connected and locally path connected, see~\cite[Theorem
3-16]{hocking_young}. Moreover, for any $\xi\in \partial \Lambda$, the space
$\partial \Lambda \setminus \{\xi\}$ has finitely many ends
by~\cite{bowditch}.

Consider $g$ as in the statement of the lemma. Its action permutes the ends
of $\partial\Lambda \setminus \{g^-\}$. Taking an iterate of $g$, we can
assume it stabilizes the ends. If $\xi$ is close to $g^-$, it is also the
case of $g\xi$. As they belong to the same end, one can join them by a small
arc $J$ that avoids $g^-$ (and $g^+$). Then $\bigcup_{n\in \Z} g^n J$ joins
$g^-$ to $g^+$, and it is invariant under $g$. However, it is not necessarily
an arc if $g^i J$ intersects $J$ in a nontrivial way for $i\not=0$. To get a
real arc, we will shorten $J$ as follows.

As $g^n J$ converges to $g^{\pm}$ when $n$ tends to $\pm\infty$, the arc $J$
can only intersect finitely many $g^i J$. Let us fix a parametrization
$u:[0,1] \to J$. The quantity
  \begin{equation*}
  \inf\{ \abs{t-s} \st s,t\in [0,1]\text{ and }\exists i\not=0, u(t)=g^i u(s)\}
  \end{equation*}
is realized by compactness (since $i$ remains bounded), for some parameters
$s,t,i$. Replacing $s,t,i$ with $t,s,-i$ if necessary, we may assume $i>0$.
As $g^-$ and $g^+$ are the only fixed points of $g^i$, we have $s\not =t$.
Let $K=u([s,t])$, this is an arc between $\eta=u(s)$ and $g^i \eta=u(t)$.
Moreover, $g^j K$ does not intersect $K$, except maybe at its endpoints for
$j=\pm i$: otherwise, there exists $x$ in the interior of $K$ such that $g^j
x$ also belongs to $K$, contradicting the minimality of $\abs{s-t}$.

It follows that $\bigcup_{n\in \Z} g^{ni} K$ is an arc from $g^-$ to $g^+$,
invariant under $g^i$.
\end{proof}

\begin{proof}[Proof of Proposition~\ref{prop:cocycle}]
Let us consider the cocycle $\bar c = c \mod \Z$. The assumption of the
proposition ensures that $\bar c(g,g^+)=0$ for all hyperbolic elements $g$.
In geometric terms, this would correspond to an assumption that the cocycle
has vanishing average on all closed orbits. Hence, we may apply a version of
Livsic's theorem, due in this context to~\cite{izumi_hyperbolic}
(Theorem~5.1). It ensures that the cocycle $\bar c$ is a coboundary: there
exists a H\"older continuous function $\bar b:\partial\Gamma \to \R/\Z$ such
that, for all $\xi \in
\partial\Gamma$, for all $g\in \Gamma$,
  \begin{equation}
  \label{eq_livsic}
  \bar c(g, \xi) = \bar b(g \xi)-\bar b(\xi).
  \end{equation}

Recall that, since the group $\Gamma$ is not virtually free, its boundary is
not totally discontinuous (see~\cite[Theorem 8.1]{kapovich_survey}). The
stabilizer of a nontrivial component $L$ of $\partial\Gamma$ is a subgroup
$\Lambda$ of $\Gamma$, quasi-convex hence hyperbolic, whose boundary is $L$
(see the discussion on top of Page~55 in~\cite{bowditch_2}).

Let us consider an infinite order element $g\in \Lambda$.
Lemma~\ref{lem:existe_axe} constructs an arc $I$ from $g^-$ to $g^+$ in
$\partial \Lambda\subset \partial\Gamma$, invariant under an iterate $g^i$ of
$g$. Replacing $g$ with $g^i$, we may assume $i=1$.

The restriction of the function $\bar b$ to the arc $I$ admits a continuous
lift $b:I\to \R$, as $I$ is simply connected. The function $F:\xi \mapsto
c(g,\xi)-b(g\xi)+b(\xi)$ is well defined on $I$, continuous, and it vanishes
modulo $\Z$ by~\eqref{eq_livsic}. Hence, it is constant. In particular, $c(g,
g^-) = F(g^-) = F(g^+) = c(g,g^+)$.
\end{proof}

In order to apply Proposition~\ref{prop:cocycle}, we will need the following
result on stable translation lengths in hyperbolic groups (\cite[Theorem
III.$\Gamma$.3.17]{bridson_haefliger}).
\begin{lem}
\label{lem_rationnel} Let $(\Gamma,d)$ be a hyperbolic group with a word
distance. Then there exists an integer $N$ such that, for any $g\in \Gamma$,
one has $N l(g)\in \Z$.
\end{lem}

The combination of Lemma~\ref{lem_explicite} and Lemma~\ref{lem_rationnel}
shows that the cocycle $c'=Nc_M/v$ satisfies $c'(g, g^+) \in \Z$ for any
hyperbolic element $g$. Moreover, this cocycle is H\"older-continuous since the
Martin cocycle $c_M$ is itself H\"older-continuous. This follows
from~\cite{izumi_hyperbolic} if $\mu$ has finite support, and
from~\cite{gouezel_infinite_support} if it has a superexponential moment.
Now, Proposition~\ref{prop:cocycle} implies the existence of a hyperbolic
element $g$ such that $c_M(g, g^+)=c_M(g, g^-)$. This is a contradiction
since $c(g, g^+) = vl(g)>0$  and $c(g, g^-)=-c(g^{-1}, g^-)=-vl(g)<0$ again
by Lemma~\ref{lem_explicite}. This concludes the proof of
Theorem~\ref{thm:hlv_indice_fini} when $\Gamma_\mu^+=\Gamma$.

If $\Gamma_\mu^+$ is a finite index subgroup of $\Gamma$, the same proof
almost works in $\Gamma_\mu^+$ to conclude that $\Gamma_\mu^+$ is virtually
free if $h=\ell v$, implying that $\Gamma$ is also virtually free. The only
difficulty is that the distance we are considering on $\Gamma_\mu^+$ is not a
word distance for a system of generators of $\Gamma_\mu^+$. However, the only
properties of the distance we have really used are:
\begin{enumerate}
\item It is hyperbolic and quasi-isometric to a word distance (to apply
    Theorem~\ref{thm:BHM}).
\item The stable translation lengths are rational numbers with bounded
    denominators.
\end{enumerate}
These two properties are clearly satisfied for the restriction of the
distance $d$ to $\Gamma_\mu^+$. Hence, the above proof also works in this
case. This completes the proof of Theorem~\ref{thm:hlv_indice_fini}. \qed

\begin{rmk}
If $\Lambda$ is a quasi-convex subgroup of a hyperbolic group $\Gamma$, then
the restriction to $\Lambda$ of a word distance on $\Gamma$ also satisfies
the above two properties. Hence, Theorem~\ref{thm:hlv_indice_fini} also holds
in $\Lambda$ for such a distance.
\end{rmk}

\section{Growth of non-distorted points in subgroups}

Our goal in this section is to prove Theorem~\ref{thm:hlv_indice_infini} on
the entropy of a random walk on an infinite index subgroup $\Lambda$ of a
hyperbolic group $\Gamma$. Since the geometry of such random walks is
complicated to describe in general, our argument is indirect: we will show
that, in any infinite index subgroup, the number of points that the random
walk effectively visits is exponentially small compared to the growth of
$\Gamma$. This is trivial if the growth $v_\Lambda = \liminf_{n\to \infty}
\frac{\log \Card{B_n\cap \Lambda}}{n}$ is strictly smaller than $v=v_\Gamma$.
When $v_\Lambda=v$, on the other hand, we will argue that the random walk
does not typically visit all of $\Lambda$, but only a subset made of
non-distorted points. To prove Theorem~\ref{thm:hlv_indice_infini}, the main
step is to show that, even when $v_\Lambda=v$, the number of such
non-distorted points is exponentially smaller than $e^{nv}$. We introduce the
notion of non-distorted points in Paragraph~\ref{subsec:distorted}, prove
this main geometric estimate in Paragraph~\ref{subsec:count_nondist}, and
apply this to random walks in Paragraph~\ref{subsec:apply_distorted}.
Paragraph~\ref{subsec:vLambda_small} is devoted to the case $v_\Lambda<v$,
where unexpected phenomena happen even in distorted subgroups.

\subsection{Non-distorted points}
\label{subsec:distorted}

There are at least two different ways to define a notion of
non-distorted point.

\begin{definition}
Let $\Gamma$ be a finitely generated group endowed with a word
distance $d=d_\Gamma$, and let $\Lambda$ be a subgroup of $\Gamma$.
\begin{itemize}
\item For $\epsilon>0$ and $M>0$, we say that $g\in \Lambda$ is $(\epsilon,
    M)$-quasi-convex if any geodesic $\gamma$ from $e$ to $g$ spends at
    least a proportion $\epsilon$ of its time in the $M$-neighborhood of
    $\Lambda$, i.e.,
  \begin{equation*}
  \Card{ \{i\in [1,\abs{g}] \st d(\gamma(i), \Lambda) \leq M\} } \geq \epsilon \abs{g}.
  \end{equation*}
    We write $\Lambda_{QC(\epsilon, M)}$ for the set of points in
     $\Lambda$ which are $(\epsilon, M)$-quasi-convex.
\item Assume additionally that $\Lambda$ is finitely generated,
    and endowed with a word distance $d_\Lambda$. For $D>0$, we
    say that $g\in \Lambda$ is $D$-undistorted if $d_\Lambda(e,g)
    \leq D d_\Gamma(e,g)$. We write $\Lambda_{UD(D)}$ for the set
    of $D$-undistorted points.
\end{itemize}
\end{definition}

Up to a change in the constants, these notions do not depend on the choice of
the distance $d$. The first definition has the advantage to work for
infinitely generated subgroups, but it may seem less natural than the second
one. If $\Lambda$ is a quasi-convex subgroup of a hyperbolic group $\Gamma$,
then all its points are $(1,M)$-quasi-convex if $M$ is large enough, and all
its points are also $D$-undistorted for large enough $D$. In the general
case, a quasi-convex point does not have to be undistorted: it may happen
that the times $i$ such that $d(\gamma(i),\Lambda) \leq M$ are all included
in $[1, \abs{g}/2]$, while between $\abs{g}/2$ and $\abs{g}$ one needs to
make a huge detour to follow $\Lambda$, making $d_\Lambda(e,g)$ much larger
than $d_\Gamma(e,g)$. On the other hand, an undistorted point is
automatically quasi-convex, at least in hyperbolic groups:
\begin{prop}
Let $\Gamma$ be a hyperbolic group, let $\Lambda$ be a finitely
generated subgroup of $\Gamma$, and let $D>0$. There exist
$\epsilon>0$ and $M>0$ such that any $D$-undistorted point is also
$(\epsilon, M)$-quasi-convex, i.e., $\Lambda_{UD(D)} \subset
\Lambda_{QC(\epsilon, M)}$.
\end{prop}
\begin{proof}
Consider $g\in \Lambda$ which is not $(\epsilon, M)$-quasi-convex, we have to
show that $d_\Lambda(e,g)$ is much bigger than $n=d_\Gamma(e,g)$. The
intuition is that, away from a $\Gamma$-geodesic from $e$ to $g$, the
progress towards $g$ is much slower by hyperbolicity.

Let us consider a geodesic from $e$ to $g$ in $\Lambda$, with length
$d_\Lambda(e,g)$. Replacing each generator of $\Lambda$ by the product of a
uniformly bounded number of generators of $\Gamma$, we obtain a path
$\gamma_\Lambda$ in the Cayley graph of $\Gamma$, remaining in the
$C_0$-neighborhood of $\Lambda$ (for some $C_0>0$) and with length
$\abs{\gamma_\Lambda}\leq C_0 d_\Lambda(e,g)$.

Let us consider a geodesic $\gamma_\Gamma$ from $e$ to $g$ for the
distance $d_\Gamma$. For each $x \in \Gamma$, we can consider its
projection $\pi(x)$ on $\gamma_\Gamma$, i.e., the point on
$\gamma_\Gamma$ that is closest to $x$ (if several points correspond,
we take the closest one to $e$). This projection is $1$-Lipschitz. In
particular, the projection of $\gamma_\Lambda$ covers the whole
geodesic $\gamma_\Gamma$. For each $x_i \in \gamma_\Gamma$, let us
consider the first point $y_i\in\gamma_\Lambda$ projecting to $x_i$.

Let us fix an integer $L$, large enough with respect to the
hyperbolicity constant of $\Gamma$. Along $\gamma_\Gamma$, let us
consider the points at distance $kL$ from $e$, i.e., $x_0=e, x_{L},
x_{2L},\dotsc, x_{mL}$ with $m=\lfloor n/L \rfloor$. In particular,
$\abs{\gamma_\Lambda} \geq \sum_i d_\Gamma(y_{iL}, y_{(i+1)L})$.
Moreover, a tree approximation shows that $d_\Gamma(y_{iL},
y_{(i+1)L}) \geq d_\Gamma(y_{iL}, x_{iL}) + L + d_\Gamma(x_{(i+1)L},
y_{(i+1)L}) - C_1$ (where $C_1$ only depends on the hyperbolicity
constant of $\Gamma$). Choosing $L\geq C_1$, we get
  \begin{equation*}
  \abs{\gamma_\Lambda} \geq \sum_{i=0}^{m} d_\Gamma(x_{iL}, y_{iL})
  \geq \sum_{i=0}^{m} (d_\Gamma(x_{iL}, \Lambda)-C_0).
  \end{equation*}
Since we assume that $g$ is not $(\epsilon,M)$-quasi-convex, the set
of indices $i$ with $d(x_i, \Lambda) \leq M$ has cardinality at most
$\epsilon n$. Taking $M\geq C_0$, the previous equation is bounded
from below by
  \begin{equation*}
  (m+1- \epsilon n) M -(m+1)C_0 \geq (n/L -\epsilon n)M - nC_0/L.
  \end{equation*}
Finally, we get
  \begin{equation*}
  d_\Lambda(e,g) \geq \abs{\gamma_\Lambda}/C_0 \geq n (1/L-\epsilon)M/C_0 - n/L.
  \end{equation*}
If $\epsilon$ is small enough and $M$ is large enough so that
$(1/L-\epsilon)M/C_0 -1/L>D$, we obtain $d_\Lambda(e,g) > Dn$, i.e., $g\notin
\Lambda_{UD(D)}$, as desired.
\end{proof}

From this point on, we will mainly work with the notion of
quasi-convex points, since counting results on such points imply
results on undistorted points by the previous proposition.

\subsection{Non-distorted points in subgroups with \texorpdfstring{$v_\Lambda=v$}{vLambda=v}}
\label{subsec:count_nondist}

In this section, we show that there are exponentially few
quasi-convex points in infinite-index subgroups of hyperbolic groups.

\begin{thm}
\label{thm:compte_peu_distordus}
Let $\Gamma$ be a nonelementary hyperbolic group endowed with a word
distance. Let $\Lambda$ be an infinite index subgroup of $\Gamma$.
Then
  \begin{equation}
  \label{eq:card_Bn}
  \Card{B_n \cap \Lambda} = o(\Card{B_n}).
  \end{equation}
Moreover, for all $\epsilon>0$ and $M>0$, there exists $\eta>0$ such that,
for all large enough $n$,
  \begin{equation}
  \label{eq:count_QC}
  \Card{ B_n \cap \Lambda_{QC(\epsilon, M)}} \leq e^{-\eta n} \Card{B_n}.
  \end{equation}
\end{thm}
One may wonder why we put the estimate~\eqref{eq:card_Bn} in the
statement of the theorem, while the main emphasis is on counting
quasi-convex points. It turns out that this estimate is not trivial,
and that its proof uses the same techniques as for the proof
of~\eqref{eq:count_QC}. To illustrate that it is not trivial, let us
remark that this estimate is not true without the hyperbolicity
assumption. For instance, in $\Gamma =\FF_2 \times \Z$ (with its
canonical generating system, and the corresponding word distance),
the infinite index subgroup $\Lambda=\FF_2$ satisfies $\Card{\Lambda
\cap B_n} / \Card{B_n} \geq c>0$.

Theorem~\ref{thm:compte_peu_distordus} is trivial if the growth rate
$v_\Lambda$ of $\Lambda$ is strictly smaller than the growth rate $v$
of $\Gamma$, since in this case $\Card{B_n \cap \Lambda}$ itself is
exponentially smaller than $\Card{B_n}$. However, this is not always
the case, even for finitely generated subgroups.

Consider for instance a compact hyperbolic $3$-manifold which fibers over the
circle, obtained as a suspension of a hyperbolic surface with a
pseudo-Anosov. Its fundamental group $\Gamma$ surjects into
$\Z=\pi_1(\Sbb^1)$. The kernel $\Lambda$ of this morphism $\phi$ is
the fundamental group of the fiber. It is finitely generated, with
infinite index, and $\Card{B_n \cap \Lambda} \sim c
\Card{B_n}/\sqrt{n}$, see~\cite{sharp_growth}.

Heuristically, one can understand in this case why there are
exponentially few quasi-convex points in $\Lambda$. Let us consider a
geodesic of length $n$ in $\Gamma$. It projects under $\phi$ to a
path in $\Z$, which behaves roughly like a random walk. In
particular, $e^{-nv} \Card{ \Sbb^n \cap \Lambda}$ behaves like the
probability that a random walk on $\Z$ comes back to the identity at
time $n$. This is of order $1/\sqrt{n}$, in accordance with the
rigorous results of~\cite{sharp_growth}. Such an element is
quasi-convex if the random walk in $\Z$ spends a big proportion of
its time close to the origin. A large deviation estimate shows that
this is exponentially unlikely.

The proof of the theorem consists in making this heuristic precise,
in the general case where the subgroup $\Lambda$ is not normal (so
that there is no morphism $\phi$ at hand). An important point in the
proof is that a hyperbolic group is automatic, i.e., there exists a
finite state automaton that recognizes a system of geodesics
parameterizing bijectively the points in the group. Counting points
in the group then amounts to a random walk on the graph of this
automaton, while counting points in $\Lambda$ amounts to a fibred
random walk, on this graph times $\Lambda \under \Gamma$. As this
space is infinite, the random walk spends most of its time outside of
finite sets, i.e., far away from $\Lambda$.

To formalize this argument, we will reduce the question to Markov
chains on graphs, where we will use the following probabilistic
lemma.
\begin{lem}
\label{lem:peu_de_retours} Consider a Markov chain $(X_n)$ on a countable set
$V$, with a stationary measure $m$ (i.e., $m(x) = \sum_y m(y) p(y,x)$ for all
$x$). Let $\tilde V$ be the set of points $x\in V$ such that $\sum_{x\to y}
m(y)=+\infty$, where we write $x\to y$ if there exists a positive probability
path from $x$ to $y$. Then, for all $x\in V$ and $x'\in \tilde V$,
  \begin{equation}
  \label{eq:pbb_tends_0}
  \Pbb_x(X_n=x') \to 0 \text{ when }n\to\infty.
  \end{equation}
Take $x\in \tilde V$ and $\epsilon>0$. There exists $\eta>0$ such that, for
all large enough $n$,
  \begin{equation}
  \label{eq:Markov_recur}
  \Pbb_x(X_n=x \text{ and $X_i$ visits $x$ at least $\epsilon n$ times in between})
  \leq e^{-\eta n}.
  \end{equation}
\end{lem}
\begin{proof}
In countable state Markov chains, a point $x$ can be either
transient, or null recurrent, or positive recurrent. Let us first
show that points in $\tilde V$ are not positive recurrent, by
contradiction. Otherwise, the points that can be reached from $x$
form an irreducible class $\boC$, which admits a stationary
probability measure $p$. The restriction of $m$ to $\boC$ is an
excessive measure. By uniqueness (see~\cite[Theorem 3.1.9]{revuz}),
the measure $m$ is proportional on $\boC$ to $p$. In particular, it
has finite mass there. This contradicts the assumption $\sum_{x\to y}
m(y)=+\infty$.

Let us now show that, for all $x\in V$ and $x'\in \tilde V$, the
probability $\Pbb_x(X_n=x')$ tends to $0$. Otherwise, conditioning on
the first visit to $x'$, we deduce that $\Pbb_{x'}(X_n=x')$ does not
tend to $0$. This implies that $x'$ is positive recurrent, a
contradiction.

Let us now prove~\eqref{eq:Markov_recur}. Consider $x\in \tilde V$,
it is either transient or null recurrent. If it is transient, the
probability $p$ to come back to $x$ is $<1$. Hence, the probability
to come back $\epsilon n$ times is bounded by $p^{\epsilon n}$, and
is therefore exponentially small as desired.

Assume now that $x$ is null recurrent: almost surely, the Markov
chain comes back to $x$, but the waiting time $\tau$ has infinite
expectation. Let $\tau_1,\tau_2,\dotsc$ be the length of the
successive excursions based at $x$. They are independent and
distributed like $\tau$, by the Markov property. The probability
in~\eqref{eq:Markov_recur} is bounded by $\Pbb( \sum_{i=1}^{\epsilon
n} \tau_i \leq n)$, which is bounded for any $M$ by $\Pbb(
\sum_{i=1}^{\epsilon n} \tau_i 1_{\tau_i \leq M} \leq n)$. The random
variables $ \tau_i 1_{\tau_i \leq M}$ are bounded, independent and
identically distributed. If $M$ is large enough, they have
expectation $>1/\epsilon$. A standard large deviation result then
shows that $\Pbb( \sum_{i=1}^{\epsilon n} \tau_i 1_{\tau_i \leq M}
\leq n)$ is exponentially small, as desired.
\end{proof}

We will also need the following technical lemma, which was explained
to us by B. Bekka.
\begin{lem}
\label{lem:indice_fini}
Let $\Lambda$ be a subgroup of a group $\Gamma$. Assume that there
exists a finite subset $B$ of $\Gamma$ such that  $B\Lambda
B=\Gamma$. Then $\Lambda$ has finite index in $\Gamma$.
\end{lem}
\begin{proof}
We have by assumption $\Gamma= \bigcup_{i,j} b_i \Lambda b_j =
\bigcup_{i,j} \Lambda_i b_i b_j$, where $\Lambda_i=b_i \Lambda
b_i^{-1}$ is a conjugate of $\Lambda$ (and has therefore the same
index). A theorem of Neumann~\cite{neumann_cosets} ensures that a
group is never a finite union of right cosets of infinite index
subgroups. Hence, one of the $\Lambda_i$ has finite index in
$\Gamma$, and so has $\Lambda$.
\end{proof}

Let $\Gamma$ be a hyperbolic group, with a finite generating set $S$.
Consider a finite directed graph $\boA = (V, E, x_*)$ with vertex set $V$,
edges $E$, a distinguished vertex $x_*$, and a labeling $\alpha:E \to S$. We
associate to any path $\gamma$ in the graph (i.e., a sequence of edges
$\sigma_0,\sigma_1,\dotsc, \sigma_{m-1}$ where the endpoint of $\sigma_i$ is
the beginning of $\sigma_{i+1}$) a path in the Cayley graph starting from the
identity and following the edges labeled $\alpha(\sigma_0)$, then
$\alpha(\sigma_1)$, and so on. The endpoint of this path is
$\alpha_*(\gamma)\coloneqq \alpha(\sigma_0) \dotsm \alpha(\sigma_{m-1})$. We
always assume that any point can be reached by a path starting at $x_*$.

A hyperbolic group is \emph{automatic} (see, for
instance,~\cite{calegari_survey}): there exists such a graph with the
following properties.
\begin{enumerate}
\item For any path $\gamma$ in the graph, the corresponding path
    $\alpha(\gamma)$ is geodesic in the Cayley graph.
\item The map $\alpha_*$ induces a bijection between the set of
    paths in the graph starting from $x_*$ and the group
    $\Gamma$.
\end{enumerate}
In particular, the paths of length $n$ in the graph originating from
$x_*$ parameterize the sphere $\Sbb^n$ of radius $n$ in the group.
The existence of such a structure makes it for instance possible to
prove that the growth series of a hyperbolic group is rational. We
will use such an automaton to count the points in the subgroup
$\Lambda$, and in particular the quasi-convex points.

We define a transition matrix $A$, indexed by $V$. By definition,
$A_{xy}$ is the number of edges from $x$ to $y$. Hence, $(A^n)_{xy}$
is the number of paths of length $n$ from $x$ to $y$. In particular,
the number of paths of length $n$ starting from $x_*$ is $\sum_y
(A^n)_{x_* y}$. Write $u$ for the line vector with $1$ at position
$x_*$ and $0$ elsewhere, and $\tilde u$ for the column vector with
$1$ everywhere. This number of paths reads $u A^n \tilde u$.
Therefore, $\Card{\Sbb^n} = u A^n \tilde u$, proving the rationality
of the growth function of the group. Let $v$ be the growth rate of
balls in $\Gamma$. It satisfies $\Card{B_n} \leq C e^{nv}$,
by~\cite{coornaert}. Hence, the spectral radius of $A$ is $e^v$, and
$A$ has no Jordan block for this maximal eigenvalue.

To understand the points of the infinite index subgroup $\Lambda$ of
$\Gamma$, we consider an extension $\boA_\Lambda$ of $\boA$, with
fibers $\Lambda \under \Gamma$. Its vertex set $V_\Lambda$ is made of
the pairs $(x, \Lambda g) \in V \times \Lambda \under \Gamma$. For
any edge $\sigma$ in $\boA$, going from $x$ to $y$ and with label
$\alpha(\sigma)$, we put for any $g\in \Gamma$ an edge in
$\boA_\Lambda$ from $(x,\Lambda g)$ to $(y, \Lambda g
\alpha(\sigma))$. A path $\gamma$ in $\boA$, from $x$ to $y$, lifts
to a path $\tilde \gamma$ in $\boA_\Lambda$ originating from $(x,
\Lambda e)$. By construction, its endpoint is $(y, \Lambda
\alpha_*(\gamma))$. This shows that the paths in the graph
$\boA_\Lambda$ remember the current right coset of $\Lambda$.

The next lemma proves that the relevant components of this fibred
graph are infinite.
\begin{lem}
\label{lem:atteint_infini}
Let $\tilde x_0=(x_0,\Lambda g_0)$ belong to $\boA_\Lambda$. Let
$\boC$ be the component of $x_0$ in $\boA$ (i.e., the set of points
that can be reached from $x_0$ and from which one can go back to
$x_0$). Let $A_\boC$ be the restriction of the matrix $A$ to the
points in $\boC$. Assume that its spectral radius $\rho(A_{\boC})$ is
equal to $e^v$. Then, starting from $\tilde x_0$ in the graph
$\boC_\Lambda$ (the restriction of $\boA_\Lambda$ to $\boC\times
\Lambda\under \Gamma$), one can reach infinitely many different
points of $\boC_\Lambda$.
\end{lem}
\begin{proof}
It suffices to show that one can reach infinitely many points whose
component in $\boC$ is $x_0$. Assume by contradiction that one can
only reach a finite number of classes $(x_0, \Lambda g_i)$.

Given $w\in \Gamma$ and $C>0$, let $Y_{w,C}$ be the set of points in $\Gamma$
that have a geodesic expression in which, for any subword $\tilde w$ of this
expression and for any $a,b$ with length at most $C$, one has $w\not=a\tilde
w b$. In other words, the points in $Y_{w,C}$ are those that never see $w$
(nor even a thickening of $w$ of size $C$) in their geodesic expressions.
Theorem~3 in~\cite{arzhantseva_lysenok} proves the existence of $C_0$ such
that, for any $w$, the quantity $\Card{B_n \cap Y_{w,C_0}}/\Card{B_n}$ tends
to $0$ (the important point is that $C_0$ does not depend on $w$).

The number of paths in $\boC$ originating from $x_0$ grows at least like
$c\Card{B_n}$ since the spectral radius of $A_{\boC}$ is $e^v$. These paths
give rise to distinct points in $\Gamma$. Hence, there exists such a path
$\gamma_0$ such that $\alpha_*(\gamma_0)\notin Y_{w,C_0}$. In particular,
there exists a subpath $\gamma_1$ such that $\alpha_*(\gamma_1)$ can be
written as $a_1 w b_1$ with $\abs{a_1}\leq C_0$ and $\abs{b_1}\leq C_0$. We
can choose a path from $x_0$ to the starting point of $\gamma_1$, with fixed
length (since $\boC$ is finite), and another path from the endpoint of
$\gamma_1$ to $x_0$. Concatenating them, we get a path $\gamma_2$ from $x_0$
to itself with $\alpha_*(\gamma_2) = a_2 w b_2$ and $\abs{a_2},\abs{b_2}\leq
C_1=C_0+2\diam(\boC)$. By assumption, $\Lambda g_0 \alpha_*(\gamma_2)$ is one
of the finitely many $\Lambda g_i$ since we are returning to $x_0$. Hence,
there exists $\lambda\in \Lambda$ such that $g_0 a_2 w b_2 = \lambda g_i$.
This shows that $w\in B \Lambda B$, where $B$ is the ball of radius $C_1+
\max_i d(e,g_i)$. As this holds for any $w$, we have proved that $B\Lambda
B=\Gamma$. By Lemma~\ref{lem:indice_fini}, this shows that $\Lambda$ has
finite index in $\Gamma$, a contradiction.
\end{proof}

\begin{lem}
\label{lem:controle_K_tilde} Let $K(n,\tilde x_0,\epsilon_0)$ denote the set
of paths in $\boA_\Lambda$ starting at a point $\tilde x_0$, of length $n$,
coming back to $\tilde x_0$ at time $n$, and spending a proportion at least
$\epsilon_0$ of the time at $\tilde x_0$. Consider $\tilde x_0\in
\boA_\Lambda$ and $\epsilon_0>0$. Then there exist $\eta>0$ and $C>0$ such
that, for all $n\in \N$,
  \begin{equation*}
  \Card{K(n,\tilde x_0,\epsilon_0)} \leq C e^{n(v-\eta)}.
  \end{equation*}
\end{lem}
\begin{proof}
Write $\tilde x_0=(x_0, \Lambda g_0)$, let $\boC$ be the component of
$x_0$ in $\boA$. If the spectral radius of the restricted transition
matrix $A_{\boC}$ is $<e^v$, we simply bound $\Card{K(n,\tilde
x_0,\epsilon_0)}$ by the number of paths in $\boC$ from $x_0$ to
itself. This is at most $\norm{A^n_\boC}$, which is exponentially
smaller than $e^{nv}$ as desired.

Assume now that $\rho(A_\boC)=e^v$. We will understand the number of
paths in $\boC$ (and in its lift $\boC_\Lambda$) in terms of a Markov
chain. The matrix $A_{\boC}$ has a unique eigenvector $q$
corresponding to the eigenvalue $e^v$, it is positive by
Perron-Frobenius's theorem. By definition, $p(x,y)=e^{-v} A_{xy}
q(y)/q(x)$ satisfies, for any $x\in\boC$,
  \begin{equation*}
  \sum_{y\in \boC} p(x,y) = \frac{e^{-v}}{q(x)} \sum A_{xy} q(y) = 1.
  \end{equation*}
This means that $p(x,y)$ is a transition kernel on $\boC$. Denote by
$(X_n)_{n\in \N}$ the corresponding Markov chain. By construction,
  \begin{equation*}
  \Pbb_x(X_n=y) = e^{-nv} (A^n)_{xy} q(y)/q(x).
  \end{equation*}
Moreover, $(A^n)_{xy}$ is the number of paths of length $n$ in $\boA$
from $x$ to $y$. Hence, up to a bounded multiplicative factor
$q(y)/q(x)$, the transition probabilities of the Markov chain $X_n$
count the number of paths in the graph $\boC$. Let $m$ denote the
unique stationary probability for the Markov chain on $\boC$.

We lift everything to $\boC_\Lambda$, assigning to an edge the
transition probability of its projection in $\boC$. The stationary
measure $m$ lifts to a stationary measure $m_\Lambda$, which is
simply the product of $m$ and of the counting measure in the
direction $\Lambda \under \Gamma$. Denoting by $X^\Lambda_n$ the
Markov chain in $\boC_\Lambda$, we have
  \begin{equation*}
  e^{-nv}\Card{K(n,\tilde x_0, \epsilon_0)} = \Pbb_{\tilde x_0}(X^\Lambda_n=\tilde x_0\text{
      and $X^\Lambda_i$ visits $\tilde x_0$ at least $\epsilon_0 n$ times in between}).
  \end{equation*}
By Lemma~\ref{lem:atteint_infini}, the Markov chain starting from
$\tilde x_0$ can reach infinitely many points. Equivalently, since
$m$ is bounded from below, it can reach a set of infinite
$m_\Lambda$-measure. Therefore, Lemma~\ref{lem:peu_de_retours}
applies, and shows that the above quantity is exponentially small.
\end{proof}

\begin{proof}[Proof of Theorem~\ref{thm:compte_peu_distordus}]
Let us first prove~\eqref{eq:count_QC}. Counting the points in $ \Sbb^n\cap
\Lambda_{QC(\epsilon,M)}$ amounts to counting the paths of length $n$ in
$\boA_\Lambda$, starting from $(x_*, \Lambda e)$ and spending a proportion at
least $\epsilon$ of their time in the finite subset $F=V \times \Lambda B_M
\subset V_\Lambda$. Such a path spends a proportion at least
$\epsilon_0=\epsilon/\Card{F}$ of its time at a given point $\tilde x\in F$.
Let $k$ and $k+m$ denote the first and last visits to $\tilde x$ (with $m\geq
\epsilon_0 n$ since there are at least $\epsilon_0 n$ visits). Such a path is
the concatenation of a path from $(x_*,\Lambda e)$ to $\tilde x$ of length
$k$ (their number is bounded by the corresponding number of paths in $\boA$,
at most $\norm{A^k} \leq C e^{kv}$), of a path in $K(m,\tilde x,\epsilon_0)$,
and of a path starting from $\tilde x$ of length $n-k-m$ (their number is
again bounded by the number of corresponding paths in $\boA$, at most $C
e^{(n-k-m)v}$). Hence, their number is at most $C e^{(n-m)v} \Card{K(m,\tilde
x, \epsilon_0)}$. Summing over the points $\tilde x \in F$, over the at most
$n$ possible values of $k$, and the values of $m$, we get the inequality
  \begin{equation*}
  \Card{\Sbb^n\cap\Lambda_{QC(\epsilon,M)}} \leq Cn e^{nv}\sum_{\tilde x\in F}
    \sum_{m=\epsilon_0 n}^n e^{-mv}\Card{K(m,\tilde x, \epsilon_0)}.
  \end{equation*}
Lemma~\ref{lem:controle_K_tilde} shows that this is exponentially
smaller than $e^{nv}$.

Let us now prove~\eqref{eq:card_Bn}, using similar arguments. A point in
$\Sbb^n\cap\Lambda$ corresponds to a path of length $n$ in $\boA_\Lambda$,
starting from $(x_*,\Lambda e)$ and ending at a point $(x,\Lambda e)$. We say
that a component $\boC$ in the graph $\boA$ is maximal if the spectral radius
of the corresponding restricted matrix $A_\boC$ is $e^v$. Since the matrix
$A$ has no Jordan block corresponding to the eigenvalue $e^v$, a path in the
graph encounters at most one maximal component. The paths in $\boA_\Lambda$
whose projection in $\boA$ spends a time $k$ in non-maximal components give
an overall contribution to $\Card{\Sbb^n\cap\Lambda}$ bounded by $C e^{(n-k)v
+ k(v-\eta)} \leq Ce^{-\eta k} \Card{B_n}$. Given $\epsilon>0$, their
contribution for $k\geq k_0(\epsilon)$ is bounded by $\epsilon\Card{B_n}$.
Hence, it suffices to control the paths for fixed $k$. Let us fix the
beginning of such a path, from $(x_*,\Lambda e)$ to a point $(x_0, \Lambda
g_0)$ where $x_0$ is in a maximal component $\boC$, and its end from $(x_1,
\Lambda g_1)$ with $x_1\in \boC$ to a point $(x,\Lambda e)$. To conclude, one
should show that the number of paths of length $n$ from $(x_0, \Lambda g_0)$
to $(x_1, \Lambda g_1)$ is $o(e^{nv})$. This follows from the probabilistic
interpretation in the proof of Lemma~\ref{lem:controle_K_tilde} and
from~\eqref{eq:pbb_tends_0}.
\end{proof}

\subsection{Non-distorted points in subgroups with \texorpdfstring{$v_\Lambda <
v$}{vLambda<v}}
\label{subsec:vLambda_small}

Let $\Lambda$ be a subgroup of a hyperbolic group $\Gamma$. Let
$v_\Lambda$ and $v_\Gamma$ be their respective growths, for a word
distance on $\Gamma$. If $v_\Lambda=v_\Gamma$,
Theorem~\ref{thm:compte_peu_distordus} proves that there is a
dichotomy:
\begin{enumerate}
\item Either $\Lambda$ is quasi-convex (equivalently, $\Lambda$
    has finite index in $\Gamma$). Then $\Card{B_n \cap \Lambda}
    \geq ce^{n v_\Lambda}$, and all points in $\Lambda$ are
    quasi-convex.
\item Or $\Lambda$ is not quasi-convex (equivalently, it has
    infinite index in $\Gamma$). Then $\Card{B_n \cap \Lambda}=
    o(e^{nv_\Lambda})$, and there are exponentially few
    quasi-convex points in $\Lambda$.
\end{enumerate}

Consider now a general subgroup $\Lambda$ with $v_\Lambda<v_\Gamma$.
If it quasi-convex, then (1) above is still satisfied: $\Card{B_n
\cap \Lambda} \geq ce^{nv_\Lambda}$ by~\cite{coornaert}, and all
points in $\Lambda$ are quasi-convex. One may ask if these properties
are equivalent, and if they characterize quasi-convex subgroups. This
question is reminiscent of a question of Sullivan in hyperbolic
geometry: Are convex cocompact groups the only ones to have finite
Patterson-Sullivan measure? Peign\'e showed in~\cite{peigne_finite_BM}
that the answer to this question is negative. His counterexamples
adapt to our situation, giving also a negative answer to our
question.

\begin{prop}
\label{prop:vLambda_small} There exists a finitely generated subgroup
$\Lambda$ of a hyperbolic group $\Gamma$ endowed with a word distance, which
is not quasi-convex, but for which $C^{-1} e^{n v_\Lambda} \leq \Card{B_n\cap
\Lambda} \leq C e^{nv_\Lambda}$. Moreover, most points of $\Lambda$ are
quasi-convex: there exist $\epsilon$ and $\eta$ such that
  \begin{equation}
  \label{eq:many_quasiconvex}
  \Card{B_n \cap \Lambda \setminus \Lambda_{QC(\epsilon,0)}} \leq Ce^{n (v_\Lambda-\eta)}.
  \end{equation}
\end{prop}
\begin{proof}
The example is the same as in~\cite{peigne_finite_BM}, but his
geometric proofs are replaced by combinatorial arguments based on
generating series.

Let $G$ be a finitely generated non-quasi-convex subgroup of a
hyperbolic group $\tilde G$ (take for instance for $\tilde G$ the
fundamental group of a hyperbolic $3$-manifold which fibers over the
circle, and for $G$ the fundamental group of the fiber of this
fibration). Let $H=\FF_k$, with $k$ large enough so that $v_H \geq
v_G$. We take $\Lambda=G*H\subset \Gamma= \tilde G* H$. It is not
quasi-convex, because of the factor $G$. Writing $v_\Lambda$ for its
growth, we claim that, for some $c>0$,
  \begin{equation}
  \label{eq:card_Sbbn}
  \Card{\Sbb^n \cap\Lambda} \sim c e^{n v_\Lambda}.
  \end{equation}
We compute with generating series. Let $F_G(z)$ be the growth series for $G$,
given by $F_G(z)=\sum_{n\geq 0} \Card{ \Sbb^n \cap G}z^n$. Likewise, we
define $F_H$ and $F_\Lambda$. Since any word in $\Lambda$ has a canonical
decomposition in terms of words in $G$ and $H$, a classical computation
(see~\cite[Prop. VI.A.4]{harpe_topics}) gives
  \begin{equation}
  \label{eq:generating}
  F_\Lambda = \frac{F_G F_H}{1-(F_G-1)(F_H-1)}.
  \end{equation}
Let $z_G=e^{-v_G} \geq z_H=e^{-v_H}$ be the convergence radii of
$F_G$ and $F_H$. At $z_H$, we have $F_H(z_H)=+\infty$, since the
cardinality of spheres in the free group is exactly of the order of
$e^{n v_H}$. When $z$ increases to $z_H$, the function
$(F_G(z)-1)(F_H(z)-1)$ takes the value $1$, at a number
$z=z_\Lambda$. Since this is the first singularity of $F_\Lambda$, we
have $z_\Lambda=e^{-v_\Lambda}$. Moreover, the function $F_\Lambda$
is meromorphic at $z_\Lambda$, with a pole of order $1$ (since the
function $(F_G-1)(F_H-1)$ has positive derivative, being a power
series with nonnegative coefficients). It follows from a simple
tauberian theorem (see, for
instance,~\cite[Theorem~IV.10]{flajolet_sedgewick}) that the
coefficients of $F_\Lambda$ behave like $c z_\Lambda^{-n}$,
proving~\eqref{eq:card_Sbbn}.

Let us estimate the number of non-quasi-convex points in $\Lambda$. Consider
a word $w\in \Lambda$ of length $n$, for instance starting with a factor in
$G$ and ending with a factor in $H$. It can be written as $g_1 h_1 g_2
h_2\dotsb h_s$. Along a geodesic from $e$ to $w$, all the words $g_1 h$ (with
$h$ prefix of $h_1$) belong to $\Lambda$. So do all the words $g_1 h_1 g_2 h$
with $h$ prefix of $h_2$, and so on. Therefore, the proportion of time that
the geodesic spends outside of $\Lambda$ is at most $\sum \abs{g_i} /n$. Such
a point in $\Lambda\setminus\Lambda_{QC(\epsilon,0)}$ satisfies $\sum
\abs{g_i} \geq (1-\epsilon) n$ and $\sum \abs{h_i} \leq \epsilon n$. Assuming
$\epsilon\leq 1/2$, this gives $\sum \abs{h_i}\leq (\epsilon/2) \sum
\abs{g_i}$. In particular, for any $\alpha>0$, we have $e^{ \alpha (\sum
\abs{g_i} - 2\epsilon^{-1} \sum \abs{h_i})} \geq 1$. Let $u_n =
\Card{\Sbb^n\cap \Lambda \setminus \Lambda_{QC(\epsilon,0)}}$, its generating
series satisfies the following equation (where we only write in details the
words starting with $G$ and ending in $H$, the other ones being completely
analogous):
  \begin{align*}
  \sum u_n z^n &
  \leq \sum_{\ell\geq 1} \sum_{a_1+b_1+a_2+\dotsb + b_\ell = n}
  e^{\alpha( \sum a_i - 2\epsilon^{-1} \sum b_i)} \Card{\Sbb^{a_1} \cap G}
  \Card{\Sbb^{b_1} \cap H} \dotsm \Card{\Sbb^{b_\ell} \cap H} z^n + \dotsc
  \\&
  = \sum_{\ell \geq 1} \bigl[(F_G(e^{\alpha} z)-1) (F_H(e^{-2\alpha \epsilon^{-1}} z)-1)\bigr]^\ell+\dotsc
  \\&
  = \frac{F_G(e^\alpha z) F_H(e^{-2\alpha \epsilon^{-1}} z)}{1-(F_G(e^\alpha z)-1)(F_H(e^{-2\alpha \epsilon^{-1}} z)-1)}.
  \end{align*}
This is the same formula as in~\eqref{eq:generating}, but the factor
$z$ has been shifted in $F_G$ and $F_H$. Choose $\alpha>0$ such that
$e^\alpha z_\Lambda < z_G$, and then $\epsilon$ small enough so that
$(F_G(e^\alpha z_\Lambda)-1)(F_H(e^{-2\alpha \epsilon^{-1}}
z_\Lambda)-1)<1$. We deduce that the series $\sum u_n z^n$ converges
for $z=z_\Lambda$, and even slightly to its right. It follows that
$u_n$ is exponentially small compared to $z_\Lambda^{-n}$. This
proves~\eqref{eq:many_quasiconvex}.
\end{proof}

\subsection{Application to random walks in infinite index subgroups}
\label{subsec:apply_distorted}

In this paragraph, we use Theorem~\ref{thm:compte_peu_distordus} to
prove Theorem~\ref{thm:hlv_indice_infini} on random walks given by a
measure $\mu$ on a hyperbolic group $\Gamma$, assuming that
$\Gamma_\mu$ has infinite index in $\Gamma$.

Before proving Theorem~\ref{thm:hlv_indice_infini}, we give another easier
result, pertaining to the case where $\mu$ has a finite moment for a word
distance on $\Gamma_\mu$ (which should be finitely generated): In this case,
the random walk typically visits undistorted points. This easy statement is
not used later on, but it gives a heuristic explanation to
Theorem~\ref{thm:hlv_indice_infini}.
\begin{lem}
Let $\Lambda$ be a finitely generated subgroup of a finitely
generated group $\Gamma$. Let $d_\Lambda$ and $d_\Gamma$ be the two
corresponding word distances. Consider a probability measure $\mu$ on
$\Lambda$, with a moment of order $1$ for $d_\Lambda$ (and therefore
for $d_\Gamma$), with nonzero drift for $d_\Gamma$. Let $X_n$ denote
the corresponding random walk. There exists $D>0$ such that $\Pbb(X_n
\in \Lambda_{UD(D)}) \to 1$.
\end{lem}
\begin{proof}
Almost surely, $d_\Gamma(e,X_n)\sim \ell_\Gamma n$, for some nonzero
drift $\ell_\Gamma$. In the same way, $d_\Lambda(e,X_n) \sim
\ell_\Lambda n$. For any $D>\ell_\Lambda/\ell_\Gamma$, we get almost
surely $d_\Lambda(e,X_n) \leq D d_\Gamma(e,X_n)$ for large enough
$n$, i.e., $X_n \in \Lambda_{UD(D)}$.
\end{proof}

This lemma readily implies Theorem~\ref{thm:hlv_indice_infini} under the
additional assumption that $\Lambda$ is finitely generated and that $\mu$ has
a moment of order $1$ for $d_\Lambda$. Indeed, for large $n$, with
probability at least $1/2$, the point $X_n$ belongs to $B_{(\ell+\epsilon)n}
\cap \Lambda_{UD(D)}$, whose cardinality is bounded by $Ce^{(\ell+\epsilon)n
(v-\eta)}$ according to Theorem~\ref{thm:compte_peu_distordus}.
Lemma~\ref{lem:charac_entropy} yields $h\leq (\ell+\epsilon)(v-\eta)$, hence
$h\leq \ell (v-\eta) < \ell v$, completing the proof.

However, the assumptions of Theorem~\ref{thm:hlv_indice_infini} are
much weaker: even when $\Lambda$ is finitely generated, it is much
more restrictive to require a moment of order $1$ on $\Lambda$ than
on $\Gamma$, precisely because the $\Gamma$-distance is smaller than
the $\Lambda$-distance on distorted points, which make up most of
$\Lambda$. The general proof will not use undistorted points (which
are not even defined when $\Lambda$ is not finitely generated), but
rather quasi-convex points: we will show that, typically, the random
walk concentrates on quasi-convex points. With the previous argument,
Theorem~\ref{thm:hlv_indice_infini} readily follows from the next
lemma.

\begin{lem}
Let $\Lambda$ be a subgroup of a hyperbolic group $\Gamma$ endowed with a
word distance $d=d_\Gamma$. Let us consider a probability measure $\mu$ on
$\Lambda$, with a moment of order $1$ for $d_\Gamma$. There exist
$\epsilon>0$ and $M>0$ such that $\Pbb(X_n \in \Lambda_{QC(\epsilon, M)})
\geq 1/2$ for large enough $n$.
\end{lem}
\begin{proof}
The lemma is trivial if $\mu$ is elementary, since all the elements
of $\Gamma_\mu\subset\Lambda$ are then quasi-convex. We may therefore
assume that $\mu$ is non-elementary.

The random walk at time $n$ is given by $X_n=g_1 \dotsm g_n$, where $g_i$ are
independent and distributed like $\mu$. We will show that most products $g_1
\dotsm g_i$ (which belong to $\Lambda$) are within distance $M$ of a geodesic
from $e$ to $X_n$ (this amounts to the classical fact that trajectories of
the random walk follow geodesics in the group), and moreover that they
approximate a proportion at least $\epsilon$ of the points on this geodesic.
This will give $X_n \in \Lambda_{QC(\epsilon, M)}$ as desired. The second
point is more delicate: we should for instance exclude the situation where,
given a geodesic $\gamma$, one has $X_n=\gamma(a(n))$ where $a(n)$ is the
smallest square larger than $n$. In this case, $X_n$ follows the geodesic
$\gamma$ at linear speed, but nevertheless the proportion of $\gamma$ it
visits tends to $0$. This behavior will be excluded thanks to the fact that,
with high probability, the jumps of the random walk are bounded.

The argument is probabilistic and formulated in terms of the bilateral
version of the random walk. On $\Omega=\Gamma^{\Z}$ with the product measure
$\Pbb=\mu^{\otimes \Z}$, let $g_n$ be the $n$-th coordinate. The $g_n$ are
independent, identically distributed, and correspond to the increments of a
random walk  $(X_n)_{n\in \Z}$ with $X_0=e$ and $X_{n}^{-1}X_{n+1}=g_{n+1}$.
Almost surely, $X_n$ converges when $n\to \pm \infty$ towards two random
variables $\xi^{\pm}\in
\partial \Gamma$, with $\xi^+ \not=\xi^-$ almost surely since these random
variables are independent and atomless. Following
Kaimanovich~\cite{kaimanovich_poisson}, denote by $S(\xi^-, \xi^+)$ the union
of all the geodesics from $\xi^-$ to $\xi^+$. Let $\pi$ be the projection on
$S(\xi^-,\xi^+)$, i.e., $\pi(g)$ is the closest point to $g$ on
$S(\xi^-,\xi^+)$. It is not uniquely defined, but two possible choices are
within distance $C_0$, for some $C_0$ only depending on $\Gamma$.

Let us choose $L>0$ large enough (how large will only depend on the
hyperbolicity constant of the space). Any measurable function is
bounded on sets with arbitrarily large measure. Hence, there exists
$K>0$ such that, with probability at least $9/10$,
  \begin{enumerate}
  \item For every $\abs{k}\geq K$, the projections $\pi(X_k)$ are
      distant from $\pi(X_0)$ by at least $L$ (and they are
      closer to $\xi^+$ if $k>0$, and to $\xi^-$ if $k<0$).
  \item We have $d(e, S(\xi^-,\xi^+))\leq K$.
  \end{enumerate}
As everything is equivariant, we deduce that, for all $i\in \Z$, the
point $X_i$ satisfies the same properties with probability at least
$9/10$, i.e.,
  \begin{equation}
  \label{eq:bonnes_props}
  d(X_i, S(\xi^-,\xi^+))\leq K \text{ and, for all $\abs{k}\geq K$, } d(\pi(X_i), \pi(X_{i+k})) \geq L.
  \end{equation}

Let $n$ be a large integer. Write $m=\lfloor n/K \rfloor$. Among the
integers $K, 2K, \dotsc, mK \leq n$, we consider the set
$I_n(\omega)$ of those $i$ such that $X_i$
satisfies~\eqref{eq:bonnes_props}. We have $\E(\Card{I_n}) \geq
m\cdot 9/10$. As $\Card{I_n}\leq m$, we get
  \begin{equation*}
  \frac{9m}{10} \leq \E(\Card{I_n}) \leq \frac{m}{10}\Pbb( \Card{I_n} < m/10) + m \Pbb( \Card{I_n} \geq m/10)
  =\frac{m}{10} + \frac{9m}{10} \Pbb( \Card{I_n} \geq m/10).
  \end{equation*}
This gives $\Pbb ( \Card{I_n} \geq m/10 ) \geq 8/9$. Let $\eta = 1/(20 K)$.
Let $\Omega_n$ be the set of $\omega$ such that $\Card{I_n(\omega)} \geq \eta
n+1$, and $X_0$ and $X_n$ satisfy~\eqref{eq:bonnes_props}, and $d(X_n, e)
\leq 2\ell n$ (where $\ell$ is the drift of $\mu$). It satisfies
$\Pbb(\Omega_n) \geq 1/2$ if $n$ is large enough. This is the set of good
trajectories for which we can control the position of many of the $X_i$.

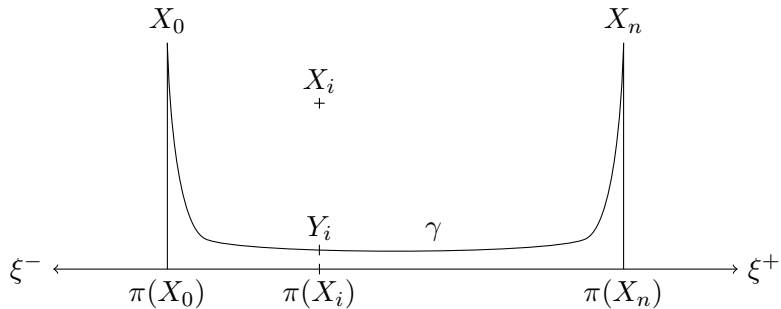
\begin{figure}
\centering
\begin{tikzpicture}
\def\tickheight{0.07}
\draw [<->] (-1.5,0) node[left] {$\xi^-$}
      -- (0,0) node[below] {$\pi(X_0)$}
      -- (2,0) node[below] {$\pi(X_i)$}
      -- (6,0) node[below] {$\pi(X_n)$}
      -- (7.5,0) node[right] {$\xi^+$};
\draw (0, 3) node[above] {$X_0$} -- (0,0)
      (6, 3) node[above] {$X_n$} -- (6,0)
      (2,2.2-\tickheight)--(2,2.2+\tickheight)
      (2-\tickheight,2.2)--(2+\tickheight,2.2)
      (2,2.2) node[above] {$X_i$}
      (2, -\tickheight) -- (2, \tickheight);
\draw [name path=MainGeod] plot[smooth, tension=0.3] coordinates {(0,3) (0.5, 0.4) (5.5, 0.4) (6,3)};
\path [name path=aboveXi] (2,0) -- (2,4);
\draw [name intersections={of=MainGeod and aboveXi}]
      (intersection-1) ++ (0,-\tickheight) -- ++(0, 2*\tickheight)
      (intersection-1) node[above] {$Y_i$};
\path [name path=aboveMiddle] (3.5,0) -- (3.5,4);
\draw [name intersections={of=MainGeod and aboveMiddle}]
      (intersection-1) node[above] {$\gamma$};
\end{tikzpicture}
\caption{The projections on $\gamma$ and $S$}
\label{fig:projections}
\end{figure}

Let $\omega\in \Omega_n$. We write $Y_i$ for a projection of $X_i$ on
a geodesic $\gamma$ from $e$ to $X_n$. Let $\tilde
I_n=I_n\setminus\{m K\}$, so that the elements of $\tilde I_n$ are at
distance at least $K$ of $0$ and $n$. As $X_0$ and $X_n$
satisfy~\eqref{eq:bonnes_props}, the projections $\pi(X_i)$ for $i\in
\tilde I_n$ are located between $\pi(X_0)$ and $\pi(X_n)$, and are at
a distance at least $L$ of these points (see
Figure~\ref{fig:projections}). If $L$ is large enough, we obtain
$d(\pi(X_i), Y_i)\leq C_1$ by hyperbolicity, where $C_1$ only depends
on $\Gamma$. This gives
  \begin{equation*}
  d(Y_i,\Lambda) \leq d(Y_i,\pi(X_i)) + d(\pi(X_i), X_i) \leq C_1+K,
  \end{equation*}
thanks to~\eqref{eq:bonnes_props} for $X_i$. When $i\not=j$ belong to
$\tilde I_n$, we have $d(\pi(X_i), \pi(X_j)) \geq L$ again thanks
to~\eqref{eq:bonnes_props}, hence $d(Y_i, Y_j) \geq L - 2C_1$. If $L$
was chosen larger than $2C_1+1$, this shows that $Y_i\not=Y_j$. We
have found along $\gamma$ at least $\Card{I_n}-1$ distinct points,
within distance $C_1+K$ of $\Lambda$. Moreover, for large enough $n$,
  \begin{equation*}
  \Card{I_n}-1 \geq \eta n  \geq 2\ell n \cdot (\eta/2\ell) \geq d(e,X_n) \cdot (\eta/2\ell).
  \end{equation*}
Let $\epsilon = \eta/2\ell$ and $M=C_1+K$. We have shown that, for $\omega
\in \Omega_n$ (whose probability is at least $1/2$), the point $X_n(\omega)$
belongs to $\Lambda_{QC(\epsilon,M)}$.
\end{proof}

\section{Construction of maximizing measures}

In this section, we prove Theorem~\ref{thm:hsurl_atteint}: Given any
finite subset $\Sigma$ in a hyperbolic group $\Gamma$, there exists a
measure $\mu_\Sigma$ maximizing the quantity $h(\mu)/\ell(\mu)$ over
all measures $\mu$ supported on $\Sigma$ with $\ell(\mu)>0$. To prove
this result, we start with a sequence of measures $\mu_i$ supported
on $\Sigma$ such that $h(\mu_i)/\ell(\mu_i)$ converges to the maximum
$M$ of these quantities. We are looking for $\mu_\Sigma$ with
$h(\mu_\Sigma)/\ell(\mu_\Sigma)=M$. Replacing $\mu_i$ with
$(\mu_i+\delta_e)/2$ (this multiplies entropy and drift by $1/2$, and
does not change their ratio) and adding $e$ to $\Sigma$, we can
always assume $\mu_i(e)\geq 1/2$, to avoid periodicity problems.

Extracting a subsequence, we can ensure that $\mu_i$ converges to a
limit probability measure $\mu$. We treat separately the two
following cases:
\begin{enumerate}
\item $\Gamma_\mu$ is non-elementary.
\item $\Gamma_\mu$ is elementary.
\end{enumerate}

Let us handle first the easy case, where $\Gamma_\mu$ is
non-elementary. In this case, the entropy and the drift are
continuous at $\mu$, by Proposition~\ref{prop:ell_continuous} and
Theorem~\ref{thm:h_continuous}, both due to Erschler and Kaimanovich
in~\cite{erschler_kaim}. Therefore, $h(\mu_i)/\ell(\mu_i)$ tend to
$h(\mu)/\ell(\mu)$, since in this case $\ell(\mu)>0$. One can thus take $\mu_\Sigma=\mu$.

\medskip

The case where $\Gamma_\mu$ is elementary is much more interesting. Let us
describe heuristically what should happen, in a simple case. We assume that
$\mu_i = (1-\epsilon) \mu + \epsilon \nu$ where $\nu$ is a fixed measure, and
$\epsilon$ tends to $0$. The random walk for $\mu_i$ can be described as
follows. At each jump, one picks $\mu$ (with probability $1-\epsilon$) or
$\nu$ (with probability $\epsilon$), then one jumps according to the chosen
measure. After time $N$, the measure $\nu$ is chosen roughly $\epsilon N$
times, with intervals of length $1/\epsilon$ in between, where $\mu$ is
chosen. Thus, $\mu_i^{*N}$ behaves roughly like
$(\mu^{*1/\epsilon}*\nu)^{\epsilon N}$.

When $\Gamma_\mu$ is finite, the measure $\mu^{*1/\epsilon}$ is
close, when $\epsilon$ is small, to the uniform measure $\pi$ on
$\Gamma_\mu$. Therefore, $\mu_i^{*N}$ is close to $(\pi *
\nu)^{\epsilon N}$. We deduce $h(\mu_i) \sim \epsilon h(\pi*\nu)$ and
$\ell(\mu_i) \sim \epsilon \ell(\pi*\nu)$. In particular,
$h(\mu_i)/\ell(\mu_i) \to h(\pi*\nu)/\ell(\pi*\nu)$. One can take
$\mu_\Sigma=\pi*\nu$.

When $\Gamma_\mu$ is infinite, it is virtually cyclic. Assuming that
$\mu$ is centered for simplicity, the walk given by
$\mu^{*1/\epsilon}$ arrives essentially at distance
$1/\sqrt{\epsilon}$ of the origin, by the central limit theorem.
Then, one jumps according to $\nu$, in a direction transverse to
$\Gamma_\mu$, preventing further cancellations. Hence, the walk given
by $(\mu^{*1/\epsilon}*\nu)^{\epsilon N}$ is at distance roughly
$\epsilon N/\sqrt{\epsilon}$ from the origin, yielding
$\ell(\mu_i)\sim \sqrt{\epsilon}$. On the other hand, each step
$\mu^{*1/\epsilon}$ only visits $1/\epsilon$ points, hence the
measure $(\mu^{*1/\epsilon}*\nu)^{\epsilon N}$ is supported by
roughly $(1/\epsilon)^{\epsilon N}$ points, yielding $h(\mu_i)\sim
\epsilon\abs{\log \epsilon}$. In particular, $h(\mu_i) =
o(\ell(\mu_i))$. This implies that $h(\mu_i)/\ell(\mu_i)$, which
tends to $0$, can not tend to the maximum $M$. Therefore, this case
can not happen.

\bigskip

The rigorous argument is considerably more delicate. One difficulty
is that $\mu_i$ does not decompose in general as $(1-\epsilon) \mu +
\epsilon \nu$: there can be in $\mu_i$ points with a very small
probability (which are not seen by $\mu$), but much larger than
$\epsilon$, the probability to visit a nonelementary subset of
$\Gamma$. These points will play an important role on the relevant
time scale, i.e., $1/\epsilon$. Hence, we have to describe the
different time scales that happen in $\mu_i$.

For each $a\in \Sigma$, we have a weight $\mu_i(a)$, which tends to
$0$ if $a$ is not in the support of $\mu$. Reordering the $a_k$ and
extracting a subsequence, we can assume that $\Sigma=\{a_1,\dotsc,
a_p\}$ with $\mu_i(a_1)\geq \dotsb \geq \mu_i(a_p)$ (and $a_1=e$).
Extracting a further subsequence, we may also assume that
$\mu_i(a_k)/\mu_i(a_{k-1})$ converges for all $k$, towards a limit in
$[0,1]$.

Let $\Gamma_k$ be the subgroup generated by $a_1,\dotsc, a_k$. We
consider the smallest $r$ such that $\Gamma_r$ is non-elementary.
Then, we consider the biggest $s<r$ such that $\mu_i(r)=o(\mu_i(s))$.
Roughly speaking, the random walk has enough time to spread
on the elementary subgroup $\Gamma_s$, before seeing $a_r$.
It turns out that the asymptotic behavior will depend on the nature of
$\Gamma_s$ (finite or virtually cyclic infinite).

We will decompose the measure $\mu_i$ as the sum of two components
$(1-\epsilon_i)\alpha_i + \epsilon_i \beta_i$, where $\epsilon_i$ tends to
$0$, the measure $\alpha_i$ mainly lives on $\Gamma_s$, and the measure
$\beta_i$ corresponds to the remaining part of $\mu_i$, on $\{a_{s+1},\dotsc,
a_p\}$. The precise construction depends on the nature of $\Gamma_s$:
\begin{itemize}
\item \emph{If $\Gamma_s$ is finite.} Let $\beta_i^{(0)}$ be the
    normalized restriction of $\mu_i$ to $\{a_{s+1},\dotsc,
    a_p\}$. To avoid periodicity problems, we rather consider
    $\beta_i =(\delta_e + \beta_i^{(0)})/2$. We decompose $\mu_i
    = (1-\epsilon_i) \alpha_i + \epsilon_i \beta_i$, where
    $\alpha_i$ is supported on $a_1,\dotsc, a_s$. By
    construction, the probability of any element in the support
    of $\alpha_i$ is much bigger than $\epsilon_i$.
\item \emph{If $\Gamma_s$ is virtually cyclic infinite.} The
    group $\Gamma_s$ contains a hyperbolic element $g_0$, with
    repelling and attracting points at infinity denoted by
    $g_0^-$ and $g_0^+$. The elements of $\Gamma_s$ all fix the
    set $\{g_0^-, g_0^+\}$. We take for $\alpha_i$ the normalized
    restriction of $\mu_i$ to those elements in $\Sigma$ that fix
    $\{g_0^-, g_0^+\}$, and for $\beta_i$ the normalized
    restriction of $\mu_i$ to the other elements. Once again, we
    can write $\mu_i = (1-\epsilon_i) \alpha_i + \epsilon_i
    \beta_i$.
\end{itemize}
In both cases, $\epsilon_i$ is comparable to the probability
$\mu_i(a_r)$, and is therefore negligible with respect to
$\mu_i(a_s)$. We will write $\mu_i = \mu_\epsilon$ (and, in the same
way, we will replace all indices $i$ with $\epsilon$, since the main
parameter is $\epsilon=\epsilon_i$). The measure $\mu_\epsilon$
converges to $\mu$ when $\epsilon$ tends to $0$, while
$\beta_\epsilon$ tends to a probability measure $\beta$, supported on
$e,a_{s+1},\dotsc, a_p$. If the measures $\mu_\epsilon$ are symmetric
to begin with, the measures $\alpha_\epsilon$ and $\beta_\epsilon$
are also symmetric by construction.

\medskip

To generate the random walk given by $\mu_\epsilon$, one can first
independently choose random measures $\rho_n$: one takes
$\rho_n=\alpha_\epsilon$ with probability $1-\epsilon$, and
$\rho_n=\beta_\epsilon$ with probability $\epsilon$. Then, one chooses
elements $g_n$ randomly according to $\rho_n$, and one multiplies them: the
product $g_1\dotsm g_n$ is distributed like the random walk given by
$\mu_\epsilon$ at time $n$.

\medskip

We will group together successive $g_k$, into blocks where the
equidistribution on $\Gamma_s$ can be seen. More precisely, denote by $t_1,
t_2,\dotsc$ the successive times where $\rho_n = \beta_\epsilon$ (and
$t_0=0$). They are stopping times, the successive differences are independent
and identically distributed, with a geometric distribution of parameter
$\epsilon$ (i.e., $\Pbb(t_1=n) = (1-\epsilon)^{n-1} \epsilon$), with mean
$1/\epsilon$. Write $L_N = g_{t_{N-1} + 1} \dotsm g_{t_N}$. By construction,
the $L_i$ are independent, identically distributed, and the random walk they
define, i.e., $L_1 \dotsm L_N$, is a subsequence of the original random walk
$g_1 \dotsm g_n$. Let $\lambda_\epsilon$ be the distribution of $L_i$ on
$\Gamma$, i.e.,
  \begin{equation*}
  \lambda_\epsilon = \sum_{n=0}^\infty (1-\epsilon)^{n} \epsilon \alpha_\epsilon^{*n} * \beta_\epsilon.
  \end{equation*}

\begin{lem}
\label{lem:exprime_gamma}
The measure $\lambda_\epsilon$ has finite first moment and finite
time one entropy. Moreover, $\ell(\mu_\epsilon) = \epsilon
\ell(\lambda_\epsilon)$ and $h(\mu_\epsilon) = \epsilon
h(\lambda_\epsilon)$.
\end{lem}
\begin{proof}
As the mean of $t_1$ is $1/\epsilon$, the random walk generated by
$\lambda_\epsilon$ is essentially the random walk generated by
$\mu_\epsilon$, but on a time scale $1/\epsilon$. This justifies
heuristically the statement.

For the rigorous proof, let us first check that $\lambda_\epsilon$
has finite first moment (and hence finite time one entropy). Since
all the measures have finite support, we have $\abs{L_1} \leq C t_1$.
Since a geometric distribution has moments of all order, the same is
true for $\abs{L_1}$.

The strong law of large numbers ensures that, almost surely, $t_N
\sim N/\epsilon$. Therefore, almost surely,
  \begin{equation*}
  \ell(\lambda_\epsilon) = \lim \frac{\abs{L_1 \dotsm L_N}}{N}
  = \lim \frac{\abs{g_1 \dotsm g_{t_N}}}{N}
  = \lim \frac{\abs{g_1 \dotsm g_{t_N}}}{t_N} \cdot \frac{t_N}{N}
  = \ell(\mu_\epsilon) \cdot 1/\epsilon.
  \end{equation*}
This proves the statement of the lemma for the drift.

For the entropy, we use the characterization of
Lemma~\ref{lem:charac_entropy}. We will show that $h(\mu_\epsilon)\leq
\epsilon h(\lambda_\epsilon)$ and $h(\mu_\epsilon)\geq \epsilon
h(\lambda_\epsilon)$. Let $K_n$ be a set of cardinality at most
$e^{(h(\mu_\epsilon)+\eta)n}$ which contains $g_1 \dotsm g_n$ with
probability at least $1/2$. Let $N=\epsilon n$. With large probability, $t_N$
is close to $n$, up to $\eta' n$ (where $\eta'$ is arbitrarily small). Hence,
with probability at least $1/3$, the point $L_1 \dotsm L_N$ belongs to the $C
\eta' n$-neighborhood of $K_n$, whose cardinality is at most
  \begin{equation*}
  \Card{K_n} \cdot e^{C' \eta' n}
  \leq e^{(h(\mu_\epsilon)+\eta+C'\eta')n}
  = e^{(h(\mu_\epsilon)+\eta+C'\eta')N/\epsilon}.
  \end{equation*}
As $\eta$ and $\eta'$ are arbitrary, this shows that $h(\lambda_\epsilon)
\leq h(\mu_\epsilon)/\epsilon$. The converse inequality is proved in the same
way.
\end{proof}

The previous lemma shows that we should understand
$\lambda_\epsilon$. We define an auxiliary probability measure
$\tilde\alpha_\epsilon$ so that $\lambda_\epsilon =
\tilde\alpha_\epsilon * \beta_\epsilon$, by
  \begin{equation}
  \label{eq:def_lambda_epsilon_tilde}
  \tilde\alpha_\epsilon = \sum_{n=0}^\infty (1-\epsilon)^n \epsilon
  \alpha_\epsilon^{*n}.
  \end{equation}
In this formula, most weight is concentrated around those $n$ of the
order of $1/\epsilon$. Hence, we have to understand the iterates of
$\alpha_\epsilon$ in time $1/\epsilon$. When $\Gamma_s$ is finite, we
will see that it has enough time to equidistribute on $\Gamma_s$
(even though $\alpha_\epsilon$ may give a very small weight to some
elements, this weight is by construction much larger than $\epsilon$,
so that $1/\epsilon$ iterates are enough to equidistribute). When
$\Gamma_s$ is virtually cyclic, we will see that the random walk has
enough time to drift away significantly from the identity.

\medskip

In both cases, we will need quantitative results on basic groups, but in
weakly elliptic cases (i.e., the transition probabilities are not bounded
from below). There are techniques to get quantitative estimates in such
settings, especially comparison techniques (due for instance to Varopoulos,
Diaconis, Saloff-Coste): one can compare weakly elliptic walks to elliptic
ones (which we understand well) thanks to Dirichlet forms arguments: these
arguments make it possible to transfer results from the latter to the former
(modulo some loss in the constants, due to the lack of ellipticity). We will
rely on such results when $\Gamma_s$ is infinite. When it is finite, such
techniques can also be used, but we will rather give a more elementary
argument.

\medskip

We start with the case where $\Gamma_s$ is finite. We need to
quantify the speed of convergence to the stationary measure in finite
groups, with the following lemma.
\begin{lem}
\label{lem:vitesse_converge} Let $\Lambda$ be a finite group. Let
$\Sigma_\Lambda\subset \Lambda$ be a generating subset (it does not have to
be symmetric). Let $\pi_\Lambda$ be the uniform measure on $\Lambda$, and let
$d(\mu,\pi_\Lambda)$ be the euclidean distance between a measure $\mu$ and
$\pi_\Lambda$ (i.e., $\left(\sum (\mu(g)-\pi_\Lambda(g))^2\right)^{1/2}$).
For any $\delta>0$, there exists $K>0$ with the following property. Let
$\eta>0$. Consider a probability measure $\mu$ on $\Lambda$ with
 $\mu(\sigma) \geq \eta$ for any $\sigma\in
\Sigma_\Lambda \cup \{e\}$. Then, for all $n>K/\eta$,
  \begin{equation*}
  d(\mu^{*n}, \pi_\Lambda) \leq \delta.
  \end{equation*}
\end{lem}
In other words, the time to see the equidistribution towards the
stationary measure is bounded by $1/\eta$, where $\eta$ is the
minimum of the transition probabilities on $\Sigma_\Lambda$.
\begin{proof}
Endow the space $\mathcal{M}(\Lambda)$ of signed measures on $\Lambda$ with
the scalar product corresponding to the quadratic form $\abs{\nu}^2=\sum
\nu(g)^2$. Denote by $H=\{\nu \st \sum \nu(g)=0\}$ the hyperplane
$\pi_\Lambda^\perp$ of zero mass measures. For any probability $\rho$, denote
by $M_\rho$ the left-convolution operator on $\mathcal{M}(\Lambda)$, that is
$M_\rho(\nu)=\rho*\nu$. Since convolution preserves mass, $H$ is
$M_{\rho}$-invariant. Let us prove that the operator norm of $M_\rho$ is
bounded by $1$. Indeed, put $u_{\rho}(g)=\sum_{h \in \Lambda}
\rho(h)\rho(hg)$, this is a probability on $\Lambda$. We have
 \begin{align*}
 \abs{M_\rho\nu}^2 & =  \sum_{g\in \Lambda} \left(M_\rho \nu(g)\right)^2
 = \sum_{(g,h_1,h_2)\in \Lambda^3} \rho(gh_1^{-1})\rho(gh_2^{-1})\nu(h_1)\nu(h_2)
 \\&
 =\sum_{(h_1,h_2)\in \Lambda^2} \nu(h_1)\nu(h_2)u_{\rho}(h_1h_2^{-1})
 = \sum_{(g,h)\in \Lambda^2} \nu(h)\nu(g^{-1}h)u_{\rho}(g)
 \\&
 \leq \sum_{g \in \Lambda} \abs{\nu}^2 u_{\rho}(g)=\abs{\nu}^2.
 \end{align*}
This proves that $\norm{M_\rho}\leq 1$. Now fix $\rho_o$ to be the uniform
probability on the set $\Sigma_{\Lambda}\cup\{e\}$. Notice that
$u_{\rho_o}(g)>0$ for any $g\in \Sigma_{\Lambda}\cup\{e\}$, since
$\rho_o(e)>0$. We claim that $M_{\rho_o}$ restricted to $H$ has an operator
norm $c<1$. Would it be not the case, there would exist $\nu\in H-\{0\}$ such
that the previous inequalities would be equalities. Thanks to the equality
case in the Cauchy-Schwarz inequality, this implies that, for any $g\in
\Sigma_{\Lambda}$, the two measures $h\mapsto \nu(h)$ and $h\mapsto
\nu(g^{-1}h)$ are positively proportional. Since their norm are equal, they
must be equal. Since $\Sigma_\Lambda$ generates $\Lambda$, $\nu$ is
$\Lambda$-invariant and belongs to $H$, so it must be zero.

\medskip

By assumption, the probability $\mu$ can be decomposed as
\[
  \mu= \eta \rho_o + (1-\eta) \nu,
\]
where $\nu$ is some probability. This implies that $M_\mu$ restricted to $H$
has operator norm at most $\eta c + (1-\eta)$. Therefore,
\[
  d(\mu^{*n}, \pi_\Lambda)=\abs{\mu^{*n} -\pi_\Lambda}=\abs{M_\mu^n (\delta_e-\pi_\Lambda)}\leq 2(1-(1-c)\eta)^n.
\]
This inequality implies the result.
\end{proof}

We can now describe the asymptotic behavior of $\mu_\epsilon$ when
the group $\Gamma_s$ is finite.
\begin{lem}
Assume that $\Gamma_s$ is finite. Define a new probability measure
$\lambda = \pi_{\Gamma_s} * \beta$ (it generates a non-elementary
subgroup). When $\epsilon$ tends to $0$, we have $h(\mu_\epsilon)\sim
\epsilon h(\lambda)$ and $\ell(\mu_\epsilon) \sim \epsilon
\ell(\lambda)$.
\end{lem}
\begin{proof}
The random variable $t_1$, being geometric of parameter $\epsilon$,
is of the order of $1/\epsilon$ with high probability (i.e., for any
$\delta>0$, there exists $u>0$ such that $\Pbb(t_1 > u/\epsilon) \geq
1-\delta$). Writing $\Sigma_s=\{a_1,\dotsc, a_s\}$ for the support of
$\alpha_\epsilon$, we have $\min_{\sigma\in \Sigma_s}
\alpha_\epsilon(\sigma) = (1-\epsilon)^{-1}\mu_\epsilon(a_s)$, which
is much bigger than $\epsilon$ by definition of $s$.
Lemma~\ref{lem:vitesse_converge} shows that the measures
$\alpha_\epsilon^{*n}$ are close to $\pi_{\Gamma_s}$ for $n\geq
u/\epsilon$. This implies that $\tilde\alpha_\epsilon$ (defined
in~\eqref{eq:def_lambda_epsilon_tilde}) converges to $\pi_{\Gamma_s}$
when $\epsilon\to 0$. As $\beta_\epsilon$ converges to $\beta$, this
shows that $\lambda_\epsilon$ converges to $\lambda$.

The support of the measure $\lambda$ contains $\Gamma_s$ and
$a_{s+1},\dotsc, a_r$ (as the support of $\beta$ contains $\{e,
a_{s+1},\dotsc, a_r\}$ by construction). Hence, $\Gamma_\lambda$
contains the non-elementary subgroup $\Gamma_r$. It follows that the
entropy and the drift are continuous at $\lambda$, by
Proposition~\ref{prop:ell_continuous} and
Theorem~\ref{thm:h_continuous}. We get $h(\lambda_\epsilon) \to
h(\lambda)$ and $\ell(\lambda_\epsilon) \to \ell(\lambda)$. With
Lemma~\ref{lem:exprime_gamma}, this completes the proof.
\end{proof}

We deduce from the lemma that $h(\mu_\epsilon)/\ell(\mu_\epsilon)$
tends to $h(\lambda)/\ell(\lambda)$. Hence, the measure
$\mu_\Sigma=\lambda$ satisfies the conclusion of the theorem, at
least in the non-symmetric case. In the symmetric case, where we are
looking for a symmetric measure $\mu_\Sigma$, the measure $\lambda =
\pi_{\Gamma_s}*\beta$ is not an answer to the problem. However,
$\lambda'=\pi_{\Gamma_s}*\beta*\pi_{\Gamma_s}$ is symmetric, and it
clearly has the same entropy and drift as $\lambda$ (since
$\pi_{\Gamma_s}*\pi_{\Gamma_s}=\pi_{\Gamma_s}$). Hence, we can take
$\mu_\Sigma=\lambda'$. This completes the proof of
Theorem~\ref{thm:hsurl_atteint} when the group $\Gamma_s$ is finite.

\begin{ex}
\label{ex:not_in_Sigma}
Let $\Gamma=\Z/2 * \Z/4$, with $\Sigma=\{a,b,b^{-1}\}$ (where $a$ is
the generator of $\Z/2$ and $b$ the generator of $\Z/4$), with the
word distance coming from $\Sigma$. \cite[Section
5.1]{mairesse_matheus} shows that the supremum over measures
supported on $\Sigma$ of $h(\mu)/\ell(\mu)$ is the growth $v$ of the
group (note that $\Gamma$ is virtually free), and that it is not
realized by a measure supported on $\Sigma$. This shows that, in
Theorem~\ref{thm:hsurl_atteint}, the fact that $\mu_\Sigma$ may need
a support larger than $\Sigma$ is not an artefact of the proof.

In this example, any symmetric measure on $\Sigma$ is of the form
$\mu_\epsilon = (1-\epsilon)\delta_a + \epsilon \beta$ where $\beta$ is
uniform on $\{b,b^{-1}\}$. The above proof shows that, when $\epsilon$ tends
to $0$, $h(\mu_\epsilon)/\ell(\mu_\epsilon)$ converges to
$h(\lambda)/\ell(\lambda)$ where $\lambda = \pi_{\Gamma_s}*\beta =
\frac{1}{2}(\delta_e + \delta_a)*\frac{1}{2}(\delta_b + \delta_{b^{-1}})$ is
the uniform measure on $\{b,b^{-1},ab,ab^{-1} \}$.
\end{ex}

It remains to treat the case where $\Gamma_s$ is virtually cyclic infinite.
Such a group surjects onto $\Z$ or $\Z \rtimes \Z/2$ (the infinite dihedral
group), with finite kernel. From the point of view of the random walk, most
things happen in the quotient. Hence, it would suffice to understand these
two groups (separating in the case of $\Z$ the centered and non-centered
cases). We will rather give direct arguments which do not use this reduction
and which avoid separating cases. Let $t\leq s$ be the smallest index such
that $\{a_1,\dotsm, a_t\}$ generates an infinite group. Let
$\eta=\eta(\epsilon)=\mu_\epsilon(a_t)$, this parameter governs the
equidistribution speed on $\Gamma_s$ (or, at least, on $\Gamma_t$, which has
finite index in $\Gamma_s$ since these two groups are virtually cyclic
infinite). We will find the asymptotics of the entropy and the drift in terms
of $\eta/\epsilon$ (which tends to infinity by definition of $s$). We start
with the entropy (for which an upper bound suffices). Note that the random
walk directed by $\alpha_\epsilon$ does not live on $\Gamma_s$, but on a
possibly bigger group since we have put in $\alpha_\epsilon$ all the points
that fix the set $\{g_0^-, g_0^+\}$ (this will be important in the control of
the drift below). Let $\tilde\Gamma_s$ be the group they generate, it is
still virtually cyclic (see, for instance,~\cite[Th\'eor\`eme 37 page
157]{ghys_hyperbolique}), and it contains $\Gamma_s$ as a finite index
subgroup.

\begin{lem}
\label{lem:majore_H}
There exists a constant $C$ such that $h(\lambda_\epsilon) \leq C
\log (\eta/\epsilon)$.
\end{lem}
\begin{proof}
Let $K$ be the group generated by $\{a_1,\dotsc, a_{t-1}\}$. It is
finite by definition of $t$. Let $\Sigma'$ be the set of points among
$a_t,\dotsc, a_p$ which stabilize $\{g_0^-, g_0^+\}$. The group
$\tilde\Gamma_s$ is generated by $K$ and $\Sigma'$. Let us consider
the associated word pseudo-distance $d'$, where we decide that
elements in $K$ have $0$ length. This pseudo-distance is
quasi-isometric to the usual distance, and it satisfies $d'(e,
xk)=d'(e,x)$ for all $x\in \tilde\Gamma_s$ and all $k\in K$.

Let us first estimate the average distance to the origin for an
element given by $\tilde\alpha_\epsilon$. We decompose
$\alpha_\epsilon$ as the average of a measure supported on
$\{a_1,\dotsc, a_{t-1}\}\subset K$, and of a measure supported on
$\Sigma'$ (the contribution of the latter has a mass $m(\epsilon)$
bounded by $(p-t+1)\eta \leq C\eta$). The measure
$\alpha_\epsilon^{*n}$ can be obtained by picking at each step one of
these two measures (according to their respective weight), and then
jumping according to a random element for this measure. When we use
the first measure, the $d'$-distance to the origin does not change by
definition. Hence, the distance to the origin is bounded by the
number of choices of the second measure. We obtain
  \begin{align*}
  \E_{\tilde\alpha_\epsilon}(d'(e,g))
  &\leq \sum_{n=0}^\infty (1-\epsilon)^{n}
  \epsilon \sum_{i=0}^n \binom{n}{i} m(\epsilon)^i (1-m(\epsilon))^{n-i} \cdot Ci
  \\&
  = C m(\epsilon)
  \sum_{n=0}^\infty (1-\epsilon)^n \epsilon \sum_{i=1}^n n \binom{n-1}{i-1} m(\epsilon)^{i-1} (1-m(\epsilon))^{n-i}
  \\&
  = C m(\epsilon) \sum_{n=0}^\infty (1-\epsilon)^n \epsilon n
  = C m(\epsilon) (1-\epsilon)/\epsilon
  \leq C \eta/\epsilon.
  \end{align*}

A measure supported on the integers with first moment $A$ has entropy
bounded by $C \log A +C$ (see, for instance, \cite[Lemma
2]{erschler_karlsson}). The proof also applies to virtually cyclic
situations (the finite thickening does not change anything).
Therefore, we get $H(\tilde\alpha_\epsilon) \leq C
\log(\eta/\epsilon)+C$.

Finally,
  \begin{equation*}
  H(\lambda_\epsilon) = H(\tilde\alpha_\epsilon*\beta_\epsilon) \leq
  H(\tilde\alpha_\epsilon) + H(\beta_\epsilon)
  \leq C \log(\eta/\epsilon)+C,
  \end{equation*}
since the support of $\beta_\epsilon$ is uniformly bounded. As
$\eta/\epsilon \to \infty$, this gives $H(\lambda_\epsilon) \leq C
\log(\eta/\epsilon)$. Finally, we estimate $h(\lambda_\epsilon) =
\inf_{n>0} H(\lambda_\epsilon^{*n})/n \leq H(\lambda_\epsilon)$ to
get the conclusion of the lemma.
\end{proof}

For the drift, we need to be more precise since we need a lower bound
to conclude. We will use a lemma giving lower bounds on the
equidistribution speed in virtually cyclic infinite groups, using comparison techniques.
\begin{lem}
\label{lem:equidistribution_virt_cyclique}
Let $\Lambda$ be a virtually cyclic infinite group. Let
$\Sigma_\Lambda\subset \Lambda$ be a finite subset generating an
infinite subgroup of $\Lambda$. There exists a constant $C$ with the
following property. Let $\eta>0$. Let $\mu$ be a probability measure
on $\Lambda$ with $\mu(e) \geq 1/2$ and $\mu(\sigma) \geq \eta$ for
any $\sigma\in \Sigma_\Lambda$. Then, for all $n\geq 1$,
  \begin{equation*}
  \sup_{g\in \Lambda} \mu^{*n}(g) \leq C (\eta n)^{-1/2}.
  \end{equation*}
\end{lem}
The interest of the lemma is that $C$ does not depend on the measure
$\mu$, and that we obtain an explicit control on $\mu^{*n}$ just in
terms of a lower bound on the transition probabilities of $\mu$.
\begin{proof}
We use the comparison method. Let $\rho$ be the uniform measure on
$e$, $\Sigma_\Lambda$ and $\Sigma_\Lambda^{-1}$. The random walk it
generates does not have to be transitive (since $\Sigma_\Lambda$ does
not necessarily generate the whole group $\Lambda$), but $\Lambda$ is
partitioned into finitely many classes where it is transitive (and
isomorphic to the random walk on the group generated by
$\Sigma_\Lambda$). Moreover, it is symmetric, and therefore
reversible for the counting measure $m$ on $\Lambda$. The Dirichlet
form associated to $\rho$ is by definition
  \begin{equation*}
  \boE_\rho(f,f) = \frac{1}{2} \sum_{x,y} \abs{f(x)-f(y)}^2 \rho(x^{-1}y),
  \end{equation*}
for any $f:\Lambda\to \C$. As $\Lambda$ has linear growth, the following Nash
inequality holds (see, for instance, \cite[Proposition 14.1]{woess}).
  \begin{equation*}
  \norm{f}_{L^2}^6 \leq C \norm{f}_{L^1}^4 \boE_\rho(f,f),
  \end{equation*}
where all norms are defined with respect to the measure $m$ on
$\Lambda$. Let $P_\mu$ be the Markov operator associated to $\mu$. It
satisfies
  \begin{equation*}
  \norm{f}_{L^2}^2 -\norm{P_\mu f}^2_{L^2} = \langle f, f\rangle - \langle P_\mu f, P_\mu f \rangle
  = \langle (I - P_\mu^* P_\mu) f, f\rangle.
  \end{equation*}
The operator $P_\mu^* P_\mu$ is the Markov operator associated to the
symmetric probability measure $\nu=\check{\mu}*\mu$, which satisfies
$\nu(\sigma) \geq \eta/2$ for $\sigma\in \Sigma_\Lambda \cup
\Sigma_\Lambda^{-1}$ and $\nu(e) \geq 1/4$ (since $\mu(e) \geq 1/2$).
Therefore, $\rho(g) \leq C \eta^{-1} \nu(g)$ for all $g$. We deduce
  \begin{align*}
  \norm{f}_{L^2}^2 -\norm{P_\mu f}^2_{L^2} &= \sum \overline{f(x)} (f(x)-f(y)) \nu(x^{-1}y)
  = \frac{1}{2} \sum \abs{f(x)-f(y)}^2 \nu(x^{-1}y)
  \\&
  \geq \frac{\eta}{2C} \sum \abs{f(x)-f(y)}^2 \rho(x^{-1}y)
  = \frac{\eta}{C} \boE_\rho(f,f).
  \end{align*}
Combining this inequality with Nash inequality, we obtain
  \begin{equation*}
  \norm{f}_{L^2}^6 \leq C\eta^{-1} \norm{f}_{L^1}^4  (\norm{f}_{L^2}^2 -\norm{P_\mu f}^2_{L^2}).
  \end{equation*}
The operator $P_\mu^*$ satisfies the same inequality, for the same
reason. Composing these inequalities, we obtain an estimate for the
norm of $P_\mu^n$ from $L^1$ to $L^\infty$ (this
is~\cite[Lemma~VII.2.6]{varopoulos_book}), of the form
  \begin{equation*}
  \norm{P_\mu^n}_{L^1 \to L^\infty} \leq (C'\eta^{-1}/n)^{1/2}.
  \end{equation*}
Applying this inequality to the function $\delta_e$, we get the
desired result.
\end{proof}

The previous lemma implies that, if $C'$ is large enough, a
neighborhood of size $(\eta n)^{1/2} / C'$ of the identity has
probability for $\mu^{*n}$ at most $1/2$. Hence, the average distance
to the origin is at least of the order of $(\eta n)^{1/2}$.

Now, we study the stationary measure for $\beta_\epsilon *
\tilde\alpha_\epsilon$ on $\partial\Gamma$. We recall that $g_0$ is a
hyperbolic element in $\Gamma_s$, fixed once and for all.
\begin{lem}
\label{lem:U_statio}
There exists a neighborhood $U$ of $\{g_0^-, g_0^+\}$ in
$\partial\Gamma$ such that the stationary measure $\nu_\epsilon$ of
$\beta_\epsilon * \tilde\alpha_\epsilon$ satisfies
$\nu_\epsilon(U)\to 0$.
\end{lem}
\begin{proof}
Let us first show that, for any neighborhood $U$ of $\{g_0^-,
g_0^+\}$, then $(\tilde\alpha_\epsilon * \delta_z)(U^c)$ tends to
$0$, uniformly in $z\in \partial \Gamma$. This is not surprising
since a typical element for $\tilde\alpha_\epsilon$ is large in the
virtually cyclic group $\tilde\Gamma_s$, and sends most points into
$U$. To make this argument rigorous, we will use
Lemma~\ref{lem:equidistribution_virt_cyclique}. The
definition~\eqref{eq:def_lambda_epsilon_tilde} shows that it suffices
to prove that $(\alpha_\epsilon^{*n}*\delta_z)(U^c)$ is small for
$n\geq u/\epsilon$.

The subgroup generated by $g_0$ has finite index in $\tilde\Gamma_s$.
Hence, any element in $\tilde\Gamma_s$ can be written as $g_0^k
\gamma_i$, for $\gamma_i$ in a finite set. Thus, the measure
$\alpha_\epsilon^{*n}$ can be written as $\sum c_n(k,i) \delta_{g_0^k
\gamma_i}$, for some coefficients $c_n(k,i)$.
Lemma~\ref{lem:equidistribution_virt_cyclique} (applied to
$\Lambda=\tilde\Gamma_s$ with $\Sigma_\Lambda=\{a_1,\dotsc, a_t\}$)
ensures that $\sup_{k,i} c_n(k,i) \leq C /(\eta n)^{1/2}$. When
$n\geq u/\epsilon$, this quantity tends to $0$ since
$\epsilon=o(\eta)$. We have
  \begin{equation*}
  (\alpha_\epsilon^{*n}*\delta_z)(U^c) = \sum_{k,i} c_n(k,i) 1( g_0^k \gamma_i z\notin U).
  \end{equation*}
As the element $g_0$ is hyperbolic, there exists $C$ such that, for
any $w\in \partial \Gamma$,
  \begin{equation*}
  \Card{ \{k\in \Z \st g_0^k w \notin U\} }\leq C.
  \end{equation*}
The uniformity in $w$ follows from the compactness of
$(\partial\Gamma\setminus \{g_0^-, g_0^+\})/\langle g_0 \rangle$. We
obtain
  \begin{equation*}
  (\alpha_\epsilon^{*n}*\delta_z)(U^c) \leq \bigl( \sup_{k,i} c_n(k,i)\bigr) \sum_i
    \Card{ \{k\in \Z \st g_0^k \gamma_i z \notin U\} }
  \leq C \sup_{k,i} c_n(k,i) \leq C/(\eta n)^{1/2}.
  \end{equation*}
This shows that $(\alpha_\epsilon^{*n}*\delta_z)(U^c)$ is small,  as
desired.

As $\tilde\alpha_\epsilon * \delta_z(U^c)$ tends to $0$ uniformly in
$z$, we deduce that $(\tilde\alpha_\epsilon
* \nu_\epsilon)(U^c)$ also tends to $0$, and therefore that
$(\tilde\alpha_\epsilon * \nu_\epsilon)(U)$ tends to $1$.

Let $A=\{g_0^-, g_0^+\}$. We claim that, for all $g$ such that $gA
\cap A\not=\emptyset$, then $gA=A$. Indeed, if $g(g_0^-)\in A$ for
instance, then $g^{-1} g_0 g$ is a hyperbolic element stabilizing
$g_0^-$. It also stabilizes $g_0^+$, by~\cite[Th\'eor\`eme 30 page
154]{ghys_hyperbolique}, i.e., $g_0g(g_0^+)=g(g_0^+)$. Hence,
$g(g_0^+)$ is a fixed point of $g_0$, i.e., $g(g_0^+)\in A$.

By definition of $\beta_\epsilon$, the finitely many elements of its
support do not fix $A$. They even satisfy $gA\cap A=\emptyset$ for
all $g$ in this support, by the previous argument. If $U$ is small
enough, we get $gU\cap U=\emptyset$, i.e., $g(U)\subset U^c$.

Finally,
  \begin{equation*}
  \nu_\epsilon(U^c) = (\beta_\epsilon*\tilde\alpha_\epsilon * \nu_\epsilon)(U^c)
  \geq (\tilde\alpha_\epsilon * \nu_\epsilon)(U),
  \end{equation*}
which tends to $1$ when $\epsilon$ tends to $0$.
\end{proof}

\begin{lem}
\label{lem:minore_fuite}
The drift $\ell(\lambda_\epsilon)$ satisfies
$\ell(\lambda_\epsilon)\geq c \cdot  (\eta/\epsilon)^{1/2}$.
\end{lem}
\begin{proof}
Let $\rho_\epsilon$ be a stationary measure for $\lambda_\epsilon$,
on the Busemann boundary $\partial_B \Gamma$. By
Proposition~\ref{prop:express_drift},
  \begin{equation*}
  \ell(\lambda_\epsilon) = \int c_B(g,\xi) \dd\rho_\epsilon(\xi) \dd\lambda_\epsilon(g),
  \end{equation*}
where $c_B(g,\xi)=h_\xi(g^{-1})$ is the Busemann cocycle. As
$\lambda_\epsilon=\tilde\alpha_\epsilon*\beta_\epsilon$, this gives
  \begin{equation*}
  \ell(\lambda_\epsilon) = \int c_B(L b,\xi) \dd\rho_\epsilon(\xi) \dd \tilde\alpha_\epsilon(L) \dd\beta_\epsilon(b).
  \end{equation*}
With the cocycle relation~\eqref{eq:cocycle}, this becomes
  \begin{equation*}
  \ell(\lambda_\epsilon)
  = \int c_B(L, b\xi)\dd\rho_\epsilon(\xi) \dd \tilde\alpha_\epsilon(L) \dd\beta_\epsilon(b)
  + \int c_B(b,\xi)\dd\rho_\epsilon(\xi) \dd \tilde\alpha_\epsilon(L) \dd\beta_\epsilon(b).
  \end{equation*}
The second integral is bounded independently of $\epsilon$ since the
support of $\beta_\epsilon$ is finite. In the first integral,
$\xi'=b\xi$ is distributed according to the measure
$\tilde\rho_\epsilon\coloneqq \beta_\epsilon*\rho_\epsilon$, which is
stationary for $\beta_\epsilon*\tilde\alpha_\epsilon$.
Lemma~\ref{lem:U_statio} implies that its projection $(\pi_B)_*
\tilde\rho_\epsilon$ on the geometric boundary, which is again
stationary for $\beta_\epsilon*\tilde\alpha_\epsilon$, gives a small
measure to a neighborhood $U$ of $\{g_0^-, g_0^+\}$.

As the limit set of $\tilde\Gamma_s$ is $\{g_0^-, g_0^+\}$, there exists a
constant $C$ such that, for all $\xi \notin \pi_B^{-1}U$ and $g\in
\tilde\Gamma_s$, we have $\abs{h_\xi(g^{-1})-d(e,g)} \leq C$. For $\xi \in
\pi_B^{-1}U$, we only use the trivial bound $h_\xi(g^{-1}) \geq - d(e,g)$,
since horofunctions are $1$-Lipschitz and vanish at the origin. We get
  \begin{align*}
  \ell(\lambda_\epsilon) &
  \geq \int_{(L,\xi)\in \Gamma\times \pi_B^{-1}U^c} d(e,L)\dd \tilde\alpha_\epsilon(L) \dd\tilde\rho_\epsilon(\xi)
    -  \int_{(L,\xi)\in \Gamma\times \pi_B^{-1}U} d(e,L)  \dd \tilde\alpha_\epsilon(L) \dd\tilde\rho_\epsilon(\xi)
    -C
  \\&
  = \paren*{ \int d(e,L)\dd \tilde\alpha_\epsilon(L)}(\tilde\rho_\epsilon(\pi_B^{-1}U^c) - \tilde\rho_\epsilon(\pi_B^{-1}U))
    -C.
  \end{align*}
For small enough $\epsilon$, we have
$\tilde\rho_\epsilon(\pi_B^{-1}U) \leq 1/4$ (and therefore
$\tilde\rho_\epsilon(\pi_B^{-1}U^c) \geq 3/4$). Moreover,
Lemma~\ref{lem:equidistribution_virt_cyclique} ensures that the
average distance to the origin for the measure
$\tilde\alpha_\epsilon$ is at least $c\cdot (\eta/\epsilon)^{1/2}$.
Hence, the previous formula completes the proof.
\end{proof}

Combining Lemmas~\ref{lem:majore_H} and~\ref{lem:minore_fuite}, we
get
  \begin{equation*}
  h(\lambda_\epsilon)/\ell(\lambda_\epsilon) \leq C \log(\eta/\epsilon) / (\eta/\epsilon)^{1/2}.
  \end{equation*}
This tends to $0$ since $\eta/\epsilon$ tends to infinity. We deduce
from Lemma~\ref{lem:exprime_gamma} that
$h(\mu_\epsilon)/\ell(\mu_\epsilon)$ tends to $0$. This is a
contradiction since we were assuming that it converges to the maximum
$M$, which is positive.

This concludes the proof of Theorem~\ref{thm:hsurl_atteint}. \qed

\medskip

The study of the case where $\Gamma_s$ is virtually cyclic infinite
gives in particular the following result.
\begin{thm}
\label{thm:infinite_limit}
Let $(\Gamma,d)$ be a metric hyperbolic group. Let $\Sigma$ be a
finite subset of $\Gamma$ which generates a non-elementary group. Let
$\mu_i$ be a sequence of measures on $\Sigma$, with $h(\mu_i)>0$,
converging to a probability measure $\mu$ such that $\Gamma_\mu$ is
infinite virtually cyclic. Then $h(\mu_i)/\ell(\mu_i)\to 0$.
\end{thm}
Note that the precise value of $\ell(\mu_i)$ depends on the choice of
the distance, but if two distances are equivalent then the associated
drifts vary within the same constants. Hence, the convergence
$h(\mu_i)/\ell(\mu_i)\to 0$ does not depend on the distance.

We recover results of Le Prince~\cite{leprince}: In any metric
hyperbolic group, there exist admissible probability measures with
$h/\ell<v$. The construction of Le Prince is rather similar to the
examples given by Theorem~\ref{thm:infinite_limit}.

\begin{ex}
We can use the above proof to also find an example where
$h(\mu_\epsilon)/\ell(\mu_\epsilon) \to 0$ although $\mu_\epsilon$
tends to a measure $\mu$ for which $\Gamma_\mu$ is finite and
nontrivial. Consider $\Gamma = \Z/2
\times \FF_2 = \{0,1\}\times \langle a,b \rangle$, endowed with the probability
measure $\mu_\epsilon$ given by
  \begin{equation*}
  \mu_\epsilon(0,e)=\mu_\epsilon(1,e)=1/2-\epsilon-\epsilon^2,\quad
  \mu_\epsilon(0,a)=\mu_\epsilon(0,a^{-1})=\epsilon,
  \quad
  \mu_\epsilon(0,b)=\mu_\epsilon(0,b^{-1})=\epsilon^2.
  \end{equation*}
The measure $\mu_\epsilon$ converges to
$\mu=(\delta_{(0,e)}+\delta_{(1,e)})/2$. With the above notations,
$\Gamma_\mu = \Z/2 \times \{e\}$ but $\Gamma_s = \Z/2 \times \langle a \rangle$
is virtually cyclic infinite (so that
$h(\mu_\epsilon)/\ell(\mu_\epsilon) \to 0$) and $\Gamma_r = \Gamma$.
\end{ex}

\section{Examples for non-symmetric measures}
\label{sec:non-symmetric}

In this section, we describe the additional difficulties that arise
if one tries to prove Theorem~\ref{thm:principal} for non-symmetric
measures. The main problem is that the random walk lives on the
subsemigroup $\Gamma_\mu^+$, which is not a subgroup any more. While
many cases can be handled with the tools we have described in this
article, one case can not be treated in this way: when the
subsemigroup $\Gamma_\mu^+$ has no nice geometric properties (it is
not quasi-convex, it is not a subgroup), but $\Gamma_\mu=\Gamma$.

Let us first show that the growth properties of such a subsemigroup
can be more complicated than what happens for subgroups. If $\Lambda$
is a subgroup of $\Gamma$, either $\Card{B_n\cap \Lambda} \asymp
e^{nv}$, or $\Card{B_n\cap \Lambda} = o(e^{nv})$ (the first case
happens if and only if $\Lambda$ has finite index in $\Gamma$, see
the discussion at the beginning of
Paragraph~\ref{subsec:vLambda_small}). Unfortunately, the behavior of
semigroups can be more complicated.

\begin{prop}
In $\FF_2$, there exists a subsemigroup $\Lambda^+$ such that
$\liminf \Card{B_n\cap \Lambda^+}/\Card{B_n}=0$ and $\limsup
\Card{B_n \cap \Lambda^+}/\Card{B_n}>0$.
\end{prop}
\begin{proof}
Let $\Sbb^n_{a,a}$ denote the geodesic words in $\FF_2=\langle
a,b\rangle$ of length $n$ which start and end with $a$. Let $n_j$ be
a sequence tending very quickly to infinity. Let $\Lambda^+$ be the
subsemigroup generated by $\bigcup \Sbb^{n_j}_{a,a}$. Then
$\Card{B_{n_j} \cap \Lambda^+} \geq c \Card{B_{n_j}}$. We claim that
  \begin{equation*}
  \Card{B_{n_j-1} \cap \Lambda^+}/\Card{B_{n_j-1}} \to 0.
  \end{equation*}
Indeed, the subsemigroup $\Lambda^+_{j-1}$ generated by
$\bigcup_{k<j}  \Sbb^{n_k}_{a,a}$ has a growth rate which is
$<e^{nv}$, since some subwords such as $b^{n_{j-1}}$ are forbidden in
this subsemigroup. Hence, if $n_j$ is large enough with respect to
$n_{j-1}$, we have $\Card{\Sbb^{n_j-1}\cap \Lambda^+}
=\Card{\Sbb^{n_j-1}\cap \Lambda^+_{j-1}} = o(e^{(n_j-1) v})$.
\end{proof}
In this example, most points in $ \Sbb^{n_j}\cap \Lambda^+$ are
introduced by $\Sbb^{n_j}_{a,a}$. This shows that $\Lambda^+$ is far
from being quasi-convex. In particular, techniques based only on
non-quasi-convexity and sub- or super-multiplicativity will never
show that $\Card{B_n\cap \Lambda^+} = o(\Card{B_n})$ for
subsemigroups.

Now, we give an example of a well-behaved measure (apart from the fact that
it is not symmetric, not admissible and not finitely supported) for which
$h=\ell v$. The construction is done in free products. The idea is to forbid
simplifications, so that we have an explicit control on the random walk at
time $n$. To enforce this behavior, we will work in a free product $\Gamma_1
* \Gamma_2$, and consider a probability measure supported on elements of the
form $g_1 g_2$ with $g_i \in \Gamma_i\setminus\{e\}$. The next statement
applies to some non virtually free hyperbolic groups, for instance the free
product of two surface groups. It also applies to some non-hyperbolic groups,
more precisely to all finitely generated groups without torsion and with
infinitely many ends, by Stallings' theorem. It would be of interest to
extend it to all groups with infinitely many ends. For this, we would need to
also handle amalgamated free products and HNN extensions.
\begin{prop}
\label{prop:nonsymmetric_equal} Let $\Gamma_1$ and $\Gamma_2$ be two
nontrivial groups, generated respectively by finite symmetric sets $S_1$ and
$S_2$. Let $\Gamma=\Gamma_1*\Gamma_2$ with the generating set $S=S_1 \cup
S_2$ and the corresponding word distance. There exists on $\Gamma$ a
(nonsymmetric, nonadmissible) probability measure $\mu$, with an exponential
moment and nonzero entropy, satisfying $h(\mu) = \ell(\mu)v$.
\end{prop}
\begin{proof}
For $i=1,2$, let $\Gamma_i^* = \Gamma_i \setminus \{e\}$. We claim
that
  \begin{equation}\label{proba}
  \sum_{g_1\in\Gamma_1^*,g_2\in\Gamma_2^*} e^{-v\abs{g_1 g_2}} = 1,
  \end{equation}
where $v$ is the growth rate of $\Gamma$.

Let $F_i(z)$ be the growth series of $\Gamma_i$, i.e., $F_i(z) =
\sum_{g\in\Gamma_i} z^{\abs{g}}$. The spheres $\Sbb^n_i\in \Gamma_i$
satisfy $\Sbb^{n+m}_i \subset \Sbb^n_i \cdot \Sbb^m_i$. Hence, the
sequence $\log \Card{\Sbb^n_i}$ is subadditive. This implies that
$\log \Card{\Sbb^n_i}/n$ converges to its infimum $v_i$, and moreover
that $\Card{\Sbb^n_i} \geq e^{nv_i}$. We deduce that the radius of
convergence of $F_i$ is $e^{-v_i}$, and moreover
$F_i(e^{-v_i})=+\infty$.

Let $F(z)$ be the growth series of $\Gamma$. As in the proof of
Proposition~\ref{prop:vLambda_small}, it is given by
  \begin{equation*}
  F(z) = \frac{F_1(z) F_2(z)}{1- (F_1(z)-1)(F_2(z)-1)}.
  \end{equation*}
Assume for instance $v_1\geq v_2$. As $F_1(e^{-v_1})=+\infty$, the
function $(F_1(z)-1)(F_2(z)-1)$ takes the value $1$ when $z$
increases to $e^{-v_1}$, at a point which is precisely the radius of
convergence $e^{-v}$ of $F$. This shows that
$(F_1(e^{-v})-1)(F_2(e^{-v})-1)=1$ . This is precisely the
equality~\eqref{proba}.

We define a probability measure $\mu$ on $\Gamma$ as follows: for
$(g_1,g_2)\in \Gamma_1^* \times \Gamma_2^*$, let
  \begin{equation*}
  \mu(g_1 g_2)=e^{-v\abs{g_1 g_2}}.
  \end{equation*}
Since there is only one way to generate the word $g_1^1 g_2^1\dotsm
g_1^n g_2^n$ using $\mu$, we have
 \begin{equation*}
 \mu^{*n}(g_1^1 g_2^1\dotsm g_1^n g_2^n) = e^{-v\sum_i\abs{g_1^i g_2^i}}.
 \end{equation*}
Denoting by $X_n$ the position of the random walk at time $n$, it
follows that $-\log \mu^{*n}(X_n)=v\abs{X_n}$. Dividing by $n$ and
letting $n$ tend to infinity, this gives $h(\mu)=\ell(\mu)v$.
\end{proof}

If one is interested in measures with finite support, one can only get the
following approximation result. It has the same flavor as
Theorem~\ref{thm:tends_to_v}, but it is both stronger since it also applies
to some non-hyperbolic groups, and weaker since the measures it produces are
not admissible nor symmetric.
\begin{prop}
Let $\Gamma_1$ and $\Gamma_2$ be two nontrivial groups, generated
respectively by finite symmetric sets $S_1$ and $S_2$. Let
$\Gamma=\Gamma_1*\Gamma_2$ with the generating set $S=S_1 \cup S_2$
and the corresponding word distance. Then
  \begin{equation*}
  \sup \big\{h(\mu)/\ell(\mu) \st
    \mu \text{ finitely supported probability measure in } \Gamma, \ell(\mu)>0\big\} = v.
  \end{equation*}
\end{prop}
\begin{proof}
Any element in $\Gamma$ can be canonically decomposed as a word in
elements of $\Gamma_1$ and $\Gamma_2$. Let $\Sbb^p_{i,j}$ be the set
of elements of length $p$ that start with an element in $\Gamma_i$
and end with an element in $\Gamma_j$. We have the decomposition
  \begin{equation*}
  \Sbb^p = \Sbb^p_{1,1} \cup \Sbb^p_{1,2} \cup \Sbb^p_{2,1} \cup \Sbb^p_{2,2}.
  \end{equation*}
One term in this decomposition has cardinality at least
$\Card{\Sbb^p}/4$. Hence, there exist $i,j$ such that $\limsup \log
\Card{\Sbb^p_{i,j}}/p = v$. Multiplying by fixed elements at the
beginning and at the end to go from $\Gamma_1$ to $\Gamma_i$, and
from $\Gamma_j$ to $\Gamma_2$, we get
  \begin{equation}
  \label{eq:limsup_OK}
  \limsup \log \Card{\Sbb^p_{1,2}}/p = v.
  \end{equation}

Let $\mu_p$ be the uniform probability measure on $\Sbb^p_{1,2}$. By
construction, there are no simplifications when one iterates $\mu_p$.
Hence, $\mu_p^{*n}$ is the uniform probability measure on
$(\Sbb^p_{1,2})^{*n}$, whose cardinality is $\Card{\Sbb^p_{1,2}}^n$.
We get $H(\mu_p^{*n})=n \log\Card{\Sbb^p_{1,2}}$ and
$L(\mu_p^{*n})=np$. Therefore, $h(\mu_p) = \log \Card{\Sbb^p_{1,2}}$
and $\ell(\mu_p)=p$, giving
  \begin{equation*}
  h(\mu_p)/\ell(\mu_p) = \log \Card{\Sbb^p_{1,2}}/p.
  \end{equation*}
Together with~\eqref{eq:limsup_OK}, this proves the proposition.
\end{proof}

\bibliography{biblio}
\bibliographystyle{amsalpha}
\end{document}